\documentclass[11pt]{amsart}

\usepackage[colorlinks=true, pdfstartview=FitV, linkcolor=blue, 
    citecolor=blue]{hyperref}
\usepackage{amsfonts, amssymb, amsmath}
\usepackage{mdframed, graphicx, color}

\usepackage{bbm, a4wide}

\usepackage{caption,subfig,tikz}
\usetikzlibrary{calc}

\newcommand{\vcoords}[2]{($sqrt(0.75)*(#1,0)+0.5*(0,#1)+(0,#2)$)}
\newcommand{\hcoords}[2]{($sqrt(12)*(#1,0)+sqrt(3)*(#2,0)+3*(0,#2)+sqrt(3)*(1,0)+(0,1)$)}

\newcommand{\ptcoords}[4]{($sqrt(12)*(#1,0)+sqrt(3)*(#2,0)+3*(0,#2)+sqrt(3)*(1,0)+(0,1)+sqrt(0.75)*(#3,0)+0.5*(0,#3)+(0,#4)$)}

\newcommand{\markpt}[2]{%
  \filldraw \vcoords{#1}{#2} circle [radius=3pt]
}
\newcommand{\markptp}[4]{%
  \filldraw[fill=red] \ptcoords{#1}{#2}{#3}{#4} circle [radius=2.5pt]
}
\newcommand{\markptq}[4]{%
  \filldraw[fill=blue] \ptcoords{#1}{#2}{#3}{#4} circle [radius=2.5pt]
}
\newcommand{\markpts}[4]{%
  \filldraw[fill=green] \ptcoords{#1}{#2}{#3}{#4} circle [radius=2.5pt]
}
\newcommand{\vctxt}[3]{%
  \draw \vcoords{#1}{#2} node {#3}
}
\newcommand{\hctxt}[3]{%
  \draw \hcoords{#1}{#2} node {#3}
}

\renewcommand{\markpt}[2]{%
  \filldraw \vcoords{#1}{#2} circle [radius=2.5pt]
}

\definecolor{lightergray}{rgb}{0.93, 0.93, 0.93}
\definecolor{mysticupper}{rgb}{0.73, 0.73, 0.58}
\definecolor{mysticlower}{rgb}{0.85, 0.86, 0.78}

\definecolor{deltacolor} {rgb}{0.80, 0.70, 0.60}
\definecolor{sigmacolor} {rgb}{0.80, 0.62, 0.49}
\definecolor{gammacolor} {rgb}{0.65, 0.65, 0.55}
\definecolor{xicolor}    {rgb}{0.60, 0.80, 0.20} 
\definecolor{picolor}    {rgb}{0.80, 0.90, 0.00} 
\definecolor{phicolor}   {rgb}{0.50, 0.85, 1.00}
\definecolor{psicolor}   {rgb}{1.00, 0.70, 0.80} 
\definecolor{thetacolor} {rgb}{1.00, 0.50, 0.60}
\definecolor{lambdacolor}{rgb}{0.60, 0.60, 1.00}

\newcommand{\edgealphaplus}[4]{%
  \draw[shift={\hcoords{#3}{#4}},rotate=#1] [rotate=#2] \vcoords{-2}{0} -- \vcoords{-2}{2}
}
\newcommand{\edgealphaminus}[4]{%
  \draw[shift={\hcoords{#3}{#4}},rotate=#1] [rotate=#2] \vcoords{-2}{0} -- \vcoords{-2}{2}
}
\newcommand{\edgebetaplus}[4]{%
  \draw[shift={\hcoords{#3}{#4}},rotate=#1] [rotate=#2] \vcoords{-2}{0} -- \vcoords{-2}{2}
}
\newcommand{\edgebetaminus}[4]{%
  \draw[shift={\hcoords{#3}{#4}},rotate=#1] [rotate=#2] \vcoords{-2}{0} -- \vcoords{-2}{2}
}
\newcommand{\edgegammaplus}[4]{%
  \draw[shift={\hcoords{#3}{#4}},rotate=#1] [rotate=#2] \vcoords{-2}{0} -- \vcoords{-2}{2}
}
\newcommand{\edgegammaminus}[4]{%
  \draw[shift={\hcoords{#3}{#4}},rotate=#1] [rotate=#2] \vcoords{-2}{0} -- \vcoords{-2}{2}
}
\newcommand{\edgedeltaplus}[4]{%
  \draw[shift={\hcoords{#3}{#4}},rotate=#1] [rotate=#2] \vcoords{-2}{0} -- \vcoords{-2}{2}
}
\newcommand{\edgedeltaminus}[4]{%
  \draw[shift={\hcoords{#3}{#4}},rotate=#1] [rotate=#2] \vcoords{-2}{0} -- \vcoords{-2}{2}
}
\newcommand{\edgeepsilonplus}[4]{%
  \draw[shift={\hcoords{#3}{#4}},rotate=#1] [rotate=#2] \vcoords{-2}{0} -- \vcoords{-2}{2}
}
\newcommand{\edgeepsilonminus}[4]{%
  \draw[shift={\hcoords{#3}{#4}},rotate=#1] [rotate=#2] \vcoords{-2}{0} -- \vcoords{-2}{2}
}
\newcommand{\edgezetaplus}[4]{%
  \draw[shift={\hcoords{#3}{#4}},rotate=#1] [rotate=#2] \vcoords{-2}{0} -- \vcoords{-2}{2}
}
\newcommand{\edgezetaminus}[4]{%
  \draw[shift={\hcoords{#3}{#4}},rotate=#1] [rotate=#2] \vcoords{-2}{0} -- \vcoords{-2}{2}
}
\newcommand{\edgeeta}[4]{%
  \draw[shift={\hcoords{#3}{#4}},rotate=#1] [rotate=#2] \vcoords{-2}{0} -- \vcoords{-2}{2}
}
\newcommand{\edgethetaplus}[4]{%
  \draw[shift={\hcoords{#3}{#4}},rotate=#1] [rotate=#2] \vcoords{-2}{0} -- \vcoords{-2}{2}
}
\newcommand{\edgethetaminus}[4]{%
  \draw[shift={\hcoords{#3}{#4}},rotate=#1] [rotate=#2] \vcoords{-2}{0} -- \vcoords{-2}{2}
}
\newcommand{\hexGamma}[3]{%
\fill[gammacolor,shift={\hcoords{#2}{#3}},rotate=#1] \vcoords{-2}{2} --
    \vcoords{-2}{0} -- \vcoords{0}{-2} -- \vcoords{2}{-2} --
    \vcoords{2}{0} -- \vcoords{0}{2} -- cycle;
  \edgealphaminus{#1}{0}{#2}{#3};
  \edgealphaplus{#1}{60}{#2}{#3};
  \edgegammaminus{#1}{120}{#2}{#3};
  \edgedeltaminus{#1}{180}{#2}{#3};
  \edgebetaplus{#1}{240}{#2}{#3};
  \edgebetaminus{#1}{300}{#2}{#3}
}
\newcommand{\hexDelta}[3]{%
  \fill[deltacolor,shift={\hcoords{#2}{#3}},rotate=#1] \vcoords{-2}{2} --
    \vcoords{-2}{0} -- \vcoords{0}{-2} -- \vcoords{2}{-2} --
    \vcoords{2}{0} -- \vcoords{0}{2} -- cycle;
  \edgegammaplus{#1}{0}{#2}{#3};
  \edgebetaplus{#1}{60}{#2}{#3};
  \edgeepsilonminus{#1}{120}{#2}{#3};
  \edgealphaplus{#1}{180}{#2}{#3};
  \edgegammaminus{#1}{240}{#2}{#3};
  \edgezetaminus{#1}{300}{#2}{#3}
}
\newcommand{\hexTheta}[3]{%
  \fill[thetacolor,shift={\hcoords{#2}{#3}},rotate=#1] \vcoords{-2}{2} --
    \vcoords{-2}{0} -- \vcoords{0}{-2} -- \vcoords{2}{-2} --
    \vcoords{2}{0} -- \vcoords{0}{2} -- cycle;
  \edgegammaplus{#1}{0}{#2}{#3};
  \edgebetaplus{#1}{60}{#2}{#3};
  \edgethetaplus{#1}{120}{#2}{#3};
  \edgebetaplus{#1}{180}{#2}{#3};
  \edgeeta{#1}{240}{#2}{#3};
  \edgebetaminus{#1}{300}{#2}{#3}
}
\newcommand{\hexLambda}[3]{%
  \fill[lambdacolor,shift={\hcoords{#2}{#3}},rotate=#1] \vcoords{-2}{2} --
    \vcoords{-2}{0} -- \vcoords{0}{-2} -- \vcoords{2}{-2} --
    \vcoords{2}{0} -- \vcoords{0}{2} -- cycle;
  \edgegammaplus{#1}{0}{#2}{#3};
  \edgebetaplus{#1}{60}{#2}{#3};
  \edgeepsilonminus{#1}{120}{#2}{#3};
  \edgealphaplus{#1}{180}{#2}{#3};
  \edgethetaminus{#1}{240}{#2}{#3};
  \edgebetaminus{#1}{300}{#2}{#3}
}
\newcommand{\hexXi}[3]{%
  \fill[xicolor,shift={\hcoords{#2}{#3}},rotate=#1] \vcoords{-2}{2} --
    \vcoords{-2}{0} -- \vcoords{0}{-2} -- \vcoords{2}{-2} --
    \vcoords{2}{0} -- \vcoords{0}{2} -- cycle;
  \edgealphaminus{#1}{0}{#2}{#3};
  \edgeepsilonplus{#1}{60}{#2}{#3};
  \edgethetaplus{#1}{120}{#2}{#3};
  \edgebetaplus{#1}{180}{#2}{#3};
  \edgeeta{#1}{240}{#2}{#3};
  \edgebetaminus{#1}{300}{#2}{#3}
}
\newcommand{\hexPi}[3]{%
  \fill[picolor,shift={\hcoords{#2}{#3}},rotate=#1] \vcoords{-2}{2} --
    \vcoords{-2}{0} -- \vcoords{0}{-2} -- \vcoords{2}{-2} --
    \vcoords{2}{0} -- \vcoords{0}{2} -- cycle;
  \edgealphaminus{#1}{0}{#2}{#3};
  \edgeepsilonplus{#1}{60}{#2}{#3};
  \edgeepsilonminus{#1}{120}{#2}{#3};
  \edgealphaplus{#1}{180}{#2}{#3};
  \edgethetaminus{#1}{240}{#2}{#3};
  \edgebetaminus{#1}{300}{#2}{#3}
}
\newcommand{\hexSigma}[3]{%
  \fill[sigmacolor,shift={\hcoords{#2}{#3}},rotate=#1] \vcoords{-2}{2} --
    \vcoords{-2}{0} -- \vcoords{0}{-2} -- \vcoords{2}{-2} --
    \vcoords{2}{0} -- \vcoords{0}{2} -- cycle;
  \edgezetaplus{#1}{0}{#2}{#3};
  \edgebetaplus{#1}{60}{#2}{#3};
  \edgeepsilonminus{#1}{120}{#2}{#3};
  \edgealphaplus{#1}{180}{#2}{#3};
  \edgegammaminus{#1}{240}{#2}{#3};
  \edgedeltaplus{#1}{300}{#2}{#3}
}
\newcommand{\hexPhi}[3]{%
  \fill[phicolor,shift={\hcoords{#2}{#3}},rotate=#1] \vcoords{-2}{2} --
    \vcoords{-2}{0} -- \vcoords{0}{-2} -- \vcoords{2}{-2} --
    \vcoords{2}{0} -- \vcoords{0}{2} -- cycle;
  \edgegammaplus{#1}{0}{#2}{#3};
  \edgebetaplus{#1}{60}{#2}{#3};
  \edgeepsilonminus{#1}{120}{#2}{#3};
  \edgeepsilonplus{#1}{180}{#2}{#3};
  \edgeeta{#1}{240}{#2}{#3};
  \edgebetaminus{#1}{300}{#2}{#3}
}
\newcommand{\hexPsi}[3]{%
  \fill[psicolor,shift={\hcoords{#2}{#3}},rotate=#1] \vcoords{-2}{2} --
    \vcoords{-2}{0} -- \vcoords{0}{-2} -- \vcoords{2}{-2} --
    \vcoords{2}{0} -- \vcoords{0}{2} -- cycle;
  \edgealphaminus{#1}{0}{#2}{#3};
  \edgeepsilonplus{#1}{60}{#2}{#3};
  \edgeepsilonminus{#1}{120}{#2}{#3};
  \edgeepsilonplus{#1}{180}{#2}{#3};
  \edgeeta{#1}{240}{#2}{#3};
  \edgebetaminus{#1}{300}{#2}{#3}
}

\newcommand{\RR}{\mathbb{R}}
\newcommand{\ZZ}{\ts\mathbb{Z}}
\newcommand{\QQ}{\mathbb{Q}}
\newcommand{\CC}{\mathbb{C}}
\newcommand{\cK}{\mathcal{K}}
\newcommand{\cO}{\mathcal{O}}
\newcommand{\cL}{\mathcal{L}}

\newcommand{\ts}{\hspace{0.5pt}}
\newcommand{\nts}{\hspace{-0.5pt}}

\newtheorem{theorem}{Theorem}
\newtheorem{coro}[theorem]{Corollary}
\theoremstyle{definition}

\newtheorem{remark}[theorem]{Remark}

\newcommand{\ee}{\ts\mathrm{e}}

\newcommand{\ii}{\ts\mathrm{i}}
\newcommand{\bs}{\boldsymbol}

\newcommand{\GAP}{\Gamma_{_\mathrm{\!\!{AP}}}}

\newcommand{\exend}{\hfill$\Diamond$}
\newcommand{\defeq}{\mathrel{\mathop:}=}

\newcommand{\myfrac}[2]{\frac{\raisebox{-2pt}{$#1$}}
  {\raisebox{0.5pt}{$#2$}}}

\title[On the long-range order of the Spectre tilings]
  {On the long-range order of the Spectre tilings}

\author{Michael Baake}
\address{Fakult\"at f\"ur Mathematik, Universit\"at Bielefeld, \newline
  \indent  Postfach 100131, 33501 Bielefeld, Germany}
\email{$\{$mbaake,gaehler,jmazac$\}$@math.uni-bielefeld.de}

\author{Franz G\"{a}hler}

\author{Jan Maz\'{a}\v{c}}

\author{Lorenzo Sadun}
\address{Department of Mathematics, Univeristy of Texas, \newline
  \indent 2515 Speedway, PMA 8.100 Austin, TX 78712, USA}
\email{sadun@math.utexas.edu}

\begin{document}

\begin{abstract}
  The Spectre is an aperiodic monotile for the Euclidean plane that is
  truly chiral in the sense that it tiles the plane without any need
  for a reflected tile. The topological and dynamical properties of
  the Spectre tilings are very similar to those of the Hat
  tilings. Specifically, the Spectre sits within a complex
  $2$-dimensional family of tilings, most of which involve two shapes
  rather than one. All tilings in the family give topologically
  conjugate dynamics, up to an overall rescaling and rotation.  They
  all have pure-point dynamical spectrum with continuous
  eigenfunctions and may be obtained from a $4:2$ dimensional
  cut-and-project scheme with regular windows of Rauzy fractal
  type. The diffraction measure of any Spectre tiling is pure point as
  well. For fixed scale and orientation, varying the shapes is MLD
  equivalent to merely varying the projection direction. These
  properties all follow from the first \v{C}ech cohomology being as
  small as it possibly could be, leaving no room for shape changes
  that alter the dynamics.
\end{abstract}

\keywords{Tiling cohomology, Dynamical spectra, Model sets,
  Deformations, Monotile}
\subjclass{52C20, 37D40, 55N05, 52C23}

\maketitle

\section{Introduction and previous results}\label{sec:intro}

In $2023$, Smith et al \cite{Hat} surprised the world with their
discovery of an aperiodic monotile, which is simply connected in
contrast to the Socolar--Taylor monotile \cite{ST}.  The \emph{Hat} is
a non-convex polykite with $14$ edges (two of which are back to back
and look like a double-length edge) and no reflection symmetry. You
can tile the Euclidean plane with isometric copies of the Hat, but the
resulting tilings cannot have any translational symmetry. All such
tilings require both rotated Hat tiles and rotated versions of the
reflected Hat, often called the \emph{anti-Hat}; see the figures in
\cite{Hat} for an illustration of the tiles and a finite patch of the
Hat tiling.

A mere two months after the discovery of the Hat, the same author team
constructed a chiral (reflection-free) analogue, now known as the
\emph{Spectre} \cite{Spectre}, by modifying one of the tiles in the
Hat family. As with the Hat, there are $12$ Spectre tiles up to
translation. The resulting tilings are non-periodic and have
statistical $6$-fold rotational symmetry, meaning that each patch
occurs in all six orientations with equal frequencies.  The $12$ tiles
can be divided into two classes. Instead of being related by
reflection, the two classes are related by rotation by $30$
degrees. That is, one only needs a single tile and rotations of that
tile by multiples of $30$ degrees to tile the plane, with all of the
resulting tilings being non-periodic.

Spectre tilings do not have statistical $12$-fold symmetry. Instead,
one group of $6$ tiles occurs with much greater frequency than the
other.  This means that there are actually two  local
  indistinguishability classes (LI classes, see \cite[Def.~5.5]{TAO})
obtained from a fixed set of $12$ Spectre tiles, each of which is a
$30$ degree rotation of the other.

In this paper, we study the dynamics and topology of the Spectre
tilings. As with the Hat tilings (compare \cite{BGS}), we will show
that there is a complex $4$-dimensional family of Spectre-like
tilings, all of which have topologically conjugate translational
dynamics, up to linear transformation.  Within this $4$-dimensional
family, there is a $2$-dimensional subfamily (including the original
Spectre) that maintains $6$-fold statistical symmetry.  These all have
topologically conjugate dynamics, up to rotation and scale. We only
need to understand the dynamics of one tiling in this family to
understand all of them. A crucial method for establishing all
  this is the \v{C}ech cohomology of the induced tiling spaces, see
  \cite{tilingsbook} for a general introduction and \cite{Hat} for the
  analogous treatment of the Hat tilings, in conjunction with the
  classification of shape changes from \cite{CS1}. 

Within the $2$-dimensional family, there is a special tiling that we
call \emph{CASPr} (for Cut-And-Symmetrically-Project). This tiling
admits a geometric substitution (an inflation rule that defines a
self-similar tiling).  Many techniques exist for studying such
self-similar tilings. Using a generalization of Solomyak's
\emph{Overlap Algorithm} \cite{Sol97,AL}, we verify that the
  CASPr tiling has pure-point dynamical spectrum. In addition, we use
  the known equivalence of pure-point diffraction and dynamical
  spectra \cite{BL} to show the same result via the cut-and-project
  method. This also implies that it can be obtained via a
cut-and-project scheme, for which we compute the lattice and the
window. More precisely, we select a Delone set that is \emph{mutually
  locally derivable} (MLD) from the CASPr tiling and admits such a
description; see \cite[Sec.~5.2]{TAO} for background on this local
version of topological conjugacy.  All other $6$-fold symmetric
tilings in the Spectre tiling family, including the original Spectre
tiling, are MLD to modified Delone sets that are obtained from
essentially the same cut-and-project scheme. One uses the same total
space and the same window to generate the different point patterns in
$\RR^4$, only with different projections from $\RR^4$ to $\RR^2$ --- a
situation that is once again analogous to that of the Hat versus the
CAP tiling \cite{BGS}. \smallskip

The structure of the paper is as follows. In
Section~\ref{sec:geometry}, we review the geometry of the Hat and
Spectre tiles and explain how to build a complex $2$-dimensional
family of Spectre-like tilings with $6$-fold statistical symmetry. In
Section~\ref{sec:top}, we then compute the \v Cech cohomology of each
space of Spectre tilings. This calculation shows that all shape
deformations of the Spectre (including those that break rotational
symmetry) are topologically conjugate to the Spectre up to linear
transformation, and that our $2$-dimensional family that respects
rotational symmetry is conjugate to the Spectre up to rotation and
scale.

In Section~\ref{sec:selfsim}, we construct the CASPr tiling and
compute its return module. In Section~\ref{sec:embed}, we construct
the cut-and-project scheme that yields the CASPr tiling and show that
it has pure-point spectrum. In Section \ref{sec:deform}, we show that
our entire $2$-dimensional family of
Spectre-like tilings are MLD to Delone sets that can be obtained by
merely rescaling, rotating, and varying the projection direction of
the CASPr tiling.

\section{Geometry of Hat and Spectre tiles}\label{sec:geometry}

As seen from the figures in \cite{Hat}, a tile in the Hat family is a
polygon with $8$ edges of length $a$ and $6$ edges of length $b$.  We
identify $\RR^2$ with $\CC$ with the real axis being \emph{vertical};
this unusual convention will prove useful when we consider the
Spectre. The basic Hat tile has $8$ edges whose displacements are $a$
times powers of $\xi = \ee^{2\pi \ii/6} = (1+ \ii \sqrt{3}\,)/2$ and
$6$ edges whose displacements are $\ii \ts b$ times powers of
$\xi$. We refer to this shape as Tile($a,b$). However, there is no
reason why $a$ and $b$ have to be real. We can consider shapes where
$a$ and $b$ are complex numbers, with the displacements along edges
being $a$ or $\ii \ts b$ times powers of $\xi$. This gives a family of
tile shapes of complex dimension 2 (real dimension 4).  Notable
elements of this family include Tile($0,1$) (the Chevron),
Tile($1,\sqrt{3}$) (the Hat), Tile($1,1$) (the Spectre),
Tile($\sqrt{3},1$) (the Turtle) and Tile($1,0$) (the Comet).  For the
Turtle, an alternate proof of aperiodicity was presented in \cite{AA}.

\begin{figure}
\centerline{\includegraphics[width=0.6\textwidth]{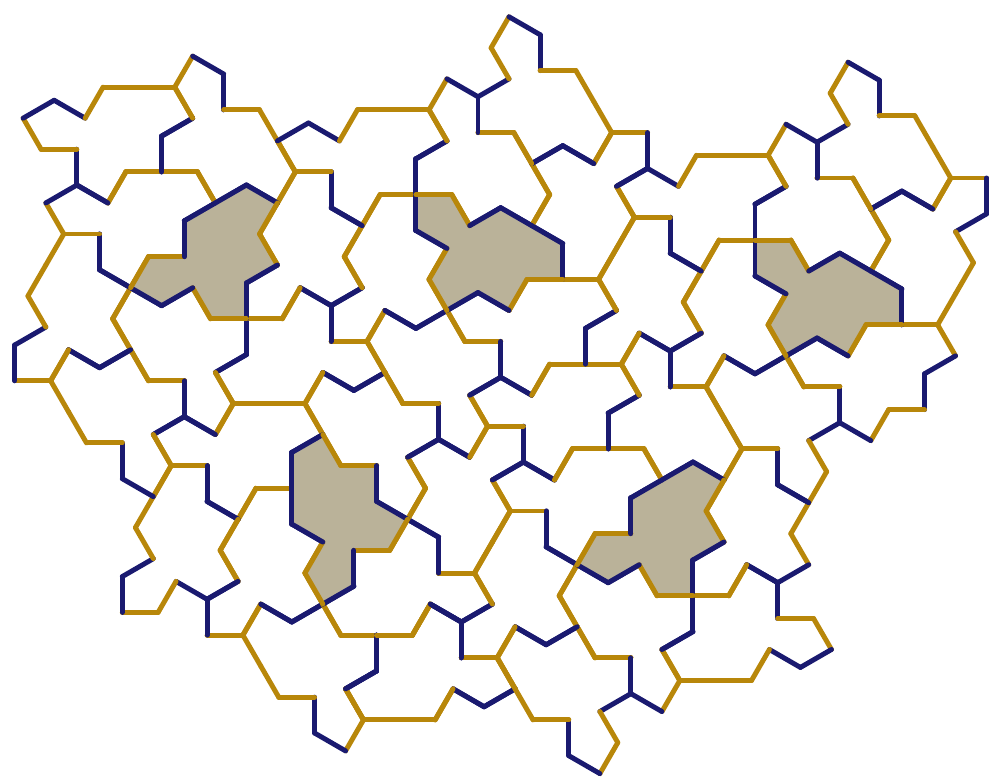}}
\caption{Patch of a Spectre tiling. Up to rotation, there is only
  one tile type. Spectres in minority (odd) orientations are shaded.
  Edges of types $a$ and $b$ are distinguished by color.
  \label{fig:spectre1}}
\end{figure}

\begin{figure}
\centerline{\includegraphics[width=0.6\textwidth]{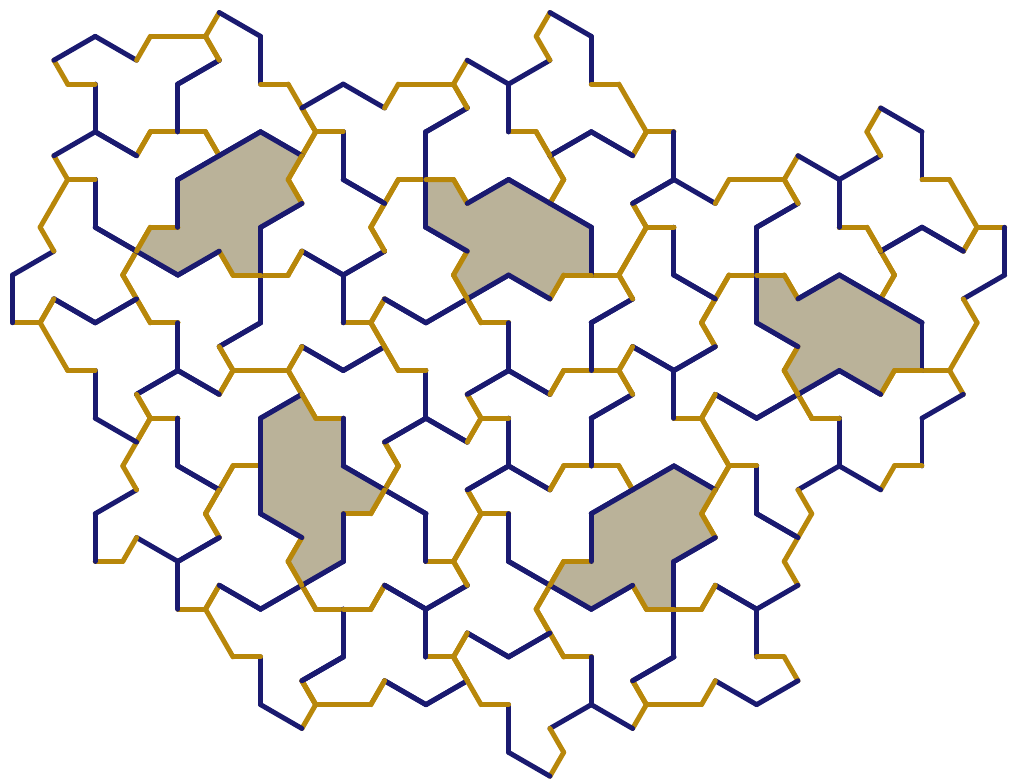}}
\caption{A combinatorially equivalent tiling to Figure~\ref{fig:spectre1},
  in which all even Spectres are replaced by Hats and all odd Spectres
  by Turtles. To achieve this, the length ratio of the two kinds of edges
  is changed from $1$ to $\sqrt{3}$.
  \label{fig:spectre2}}
\end{figure}

A Hat-like tiling is made from Tile($a,b$) in $6$ orientations and the
reflection of Tile($\bar a$, $\bar b$) in $6$ orientations.  (We
define our standard reflection $m$ to be horizontal:
$m(x,y) = (-x,y)$. With our identification of $\RR^2$ with $\CC$, this
is equivalent to complex conjugation.) If $a/b$ is real, all of the
tiles are isometric and we have a monotile. However, even if $a/b$ is
complex (in which case Tile($\bar a$, $\bar b$) is not isometric to
Tile($a,b$)), we still have a well-defined space of tilings.

A Spectre-like tiling is made from Tile($a,b$) rotated by even
multiples of $30$ degrees and Tile($b,a$) rotated by odd multiples of
$30$ degrees. For the ratio $a/b=1$, we have only one tile type, the
Spectre. A patch of such a Spectre tiling is shown in
Figure~\ref{fig:spectre1}. If we vary the ratio $a/b$ to $\sqrt{3}$,
we get two tile types: the even Spectres become Hats and the odd
Spectres become Turtles, as shown in the combinatorially equivalent
patch in Figure~\ref{fig:spectre2}. Indeed, the proof of aperiodicity
in \cite{Spectre} is based on analyzing a tiling by Hats and Turtles
that is combinatorially equivalent to a tiling by even and odd
Spectres. The version with Hats and Turtles has the advantage that the
tiles are located on an underlying hexagonal lattice. For the Spectre
tiling, this is not the case. 

It is immediate that the Hat-Turtle tilings (with $a/b=\sqrt{3}$) and
the Spectre tilings (with $a/b=1$) are related by a \emph{shape
  change} of the tiles, maintaining the combinatorics of the
tilings. Such shape changes are classified, up to local
  equivalence, by $\check H^1(\Omega, \RR^2)$, the first \v Cech
  cohomology of the tiling space with values in $\RR^2$, see
  \cite{CS1}, or equivalently by $\check H^1(\Omega, \CC)$, with a
  subgroup of $\check H^1(\Omega, \CC)$ parametrizing shape changes
  that are topological conjugacies.  The calculations of
  Section~\ref{sec:top} will imply that the translation actions of the
  two tiling dynamical systems are topologically conjugate.

Note that Tile($a,b$) is only isometric to Tile($b,a$) when $a=b$.
While there is a large family of Spectre-like tilings, there is
essentially only one chiral monotile, namely the actual Spectre
Tile($1,1$).  We note in passing that it is possible to make a
periodic tiling from Tile($1,1$) and its reflection.  To eliminate
this possibility, the authors of \cite{Spectre} added edge markings to
Tile($1,1$) that prevent Spectres and anti-Spectres from fitting
together. If one simply defines the prototile set to include tiles but
not anti-tiles, such decorations are unnecessary.

\section{Substitutive structure and cohomology}\label{sec:top}

In the Spectre tilings, as in the Hat tilings, tiles aggregate into
clusters called \emph{meta-tiles}.  These meta-tiles assemble into
larger clusters with the same combinatorial structure as the
meta-tiles, which assemble into still larger structures, and so
on. This fits well into the fusion setting \cite{FS}. That is, both
the Hat and the Spectre tilings are (combinatorial) substitution
tilings, and so can be studied with a variety of well-established
topological and dynamical tools. There are technical differences, of
course, and the details are more complicated with the Spectre than
with the Hat, but the overall picture is very similar.

The substitutive structure of the Spectre tiling was described in
\cite{Spectre}.  Imagine a Spectre tiling in which even Spectres
outnumber odd Spectres. A combinatorially equivalent Hat-Turtle tiling
would feature isolated Turtles in a sea of Hats. There is one and only
one way for a Turtle to be surrounded by Hats.  Combining the Turtle
with a particular one of its surrounding Hats yields a shape called a
\emph{Mystic}. The corresponding tiling can then be described as a
substitution involving two basic units: Mystics, and Hats that are not
part of Mystics.

However, there are two major complications. The first is that the
substituted Hats and Mystics are not made up of (rotated) Hats and
Mystics. Instead, they are made up of (rotated and) \emph{reflected}
Hats and Mystics, also known as anti-Hats and anti-Mystics.  Likewise,
a substituted anti-Hat or anti-Mystic is a cluster of ordinary Hats
and Mystics. However, we are interested in tilings that only involve
Hats and Mystics, not anti-Hats or anti-Mystics{\ts}! To go from Hats
and Mystics to Hats and Mystics, we must apply the substitution twice,
getting much larger clusters that are cumbersome to work with.

Put another way, we have a substitution $\sigma$ that takes Hats and
Mystics to anti-Hats and anti-Mystics. The reflection of this rule
gives a substitution $\sigma^*$ that takes anti-Hats and anti-Mystics
to Hats and Mystics. The total substitution on all four kinds of tiles
is given by the matrix
\[
  \begin{pmatrix} 0 & \sigma^* \\ \sigma & 0 \end{pmatrix} ,
\]
which swaps the unreflected and reflected sectors. The square of the
substitution is then
\[
  \begin{pmatrix} \sigma^* \nts \sigma & 0 \\
    0 & \sigma \ts \sigma^* \end{pmatrix},
\]
which preserves sectors. Since we are interested only in unreflected
tiles, we need to study $\sigma^* \sigma$, which we accomplish by
studying $\sigma$ in a suitable way.

The second complication is that the substitution of Hats and Mystics
does not \emph{force the border} in the sense of \cite{kel}.  In their
proof of aperiodicity, Smith et al.~\cite{Spectre} introduced a
substitution involving nine meta-tiles. One is a substituted Mystic,
denoted $\Gamma$. The other eight are distinct collared versions of
the substituted Hat, denoted by $\Delta$, $\Theta$, $\Lambda$, $\Xi$,
$\Pi$, $\Sigma$, $\Phi$ and $\Psi$. With this additional structure,
the substitution does force the border, so we can apply the methods of
Anderson and Putnam \cite{AP} to compute the \v{C}ech cohomology,
  as we now do.

These nine meta-tiles all are combinatorial hexagons that meet edge to
edge.  For the Hat-Turtle tiling, they are shown in
\cite[Fig.~4.1]{Spectre}.  The tiling with the Hats and Turtles
composed to meta-tiles is MLD to the plain Hat-Turtle tiling.  The
meta-tile tiling of combinatorial hexagons is then related by a shape
change to a combinatorially equivalent tiling of \emph{regular}
hexagons, so that, up to MLD conjugacies, the latter is also related
by a shape change to the original Spectre tiling.

There are $8$ kinds of edges of the meta-tiles, labeled $\alpha$,
$\beta$, $\gamma$, $\delta$, $\epsilon$, $\zeta$, $\theta$, and
$\eta$.  We use them in this order for reasons that will become clear
shortly.  There are three kinds of vertices, labeled $p$, $q$ and
$s$.  The geometric and combinatorial information is distilled in
Figure~\ref{fig:combhex}, showing each meta-tile as a regular hexagon
with labeled edges and colored vertices.

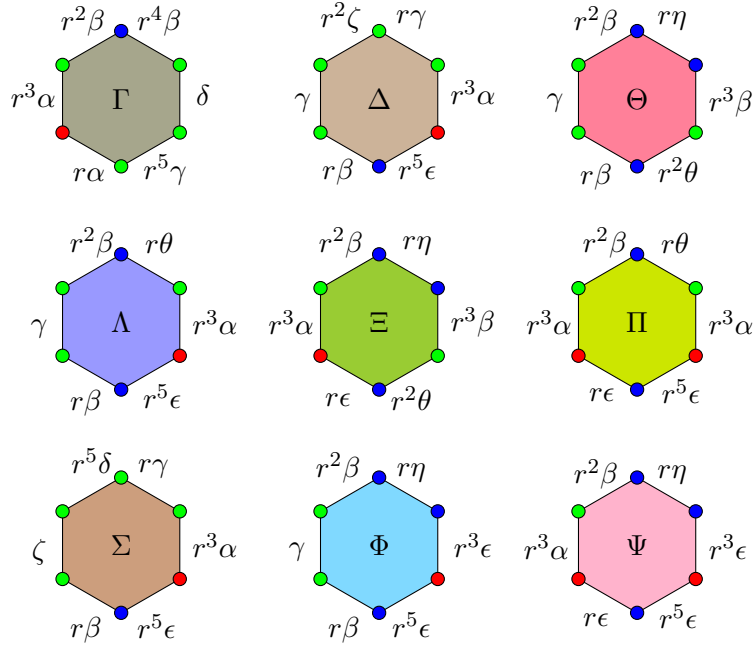
\begin{figure}[ht!]
\begin{center}
\begin{tikzpicture}[x=4.5mm,y=4.5mm]
 \hexGamma{0}{0}{0};
  \hctxt{0}{0}{$\Gamma$};
  \vctxt{-1}{1.6}{$r^3\alpha$};
  \vctxt{0.9}{-1.7}{$r\alpha$};
  \vctxt{3.5}{-2.85}{$r^5\gamma$};
  \vctxt{4.8}{-1.3}{$\delta$};
  \vctxt{3.3}{1.65}{$r^4\beta$};
  \vctxt{0.7}{2.95}{$r^2\beta$};
  \markptp{0}{0}{-2}{0};
  \markptq{0}{0}{0}{2};
  \markpts{0}{0}{2}{-2};
  \markpts{0}{0}{0}{-2};
  \markpts{0}{0}{2}{0};
  \markpts{0}{0}{-2}{2};
 \hexDelta{0}{2.2}{0};
  \hctxt{2.2}{0}{$\Delta$};
  \vctxt{8.2}{-3.22}{$\gamma$};
  \vctxt{9.4}{-5.87}{$r\beta$};
  \vctxt{12.1}{-7.06}{$r^5\epsilon$};
  \vctxt{14.0}{-5.8}{$r^3\alpha$};
  \vctxt{11.9}{-2.6}{$r\gamma$};
  \vctxt{9.5}{-1.35}{$r^2\zeta$};
  \markpts{2.2}{0}{-2}{0};
  \markptp{2.2}{0}{2}{-2};
  \markpts{2.2}{0}{0}{2};
  \markptq{2.2}{0}{0}{-2};
  \markpts{2.2}{0}{2}{0};
  \markpts{2.2}{0}{-2}{2};
 \hexTheta{0}{4.4}{0};
  \hctxt{4.4}{0}{$\Theta$};
  \vctxt{16.8}{-7.52}{$\gamma$};
  \vctxt{18.2}{-10.4}{$r\beta$};
  \vctxt{21.0}{-11.6}{$r^2\theta$};
  \vctxt{22.8}{-10.5}{$r^3\beta$};
  \vctxt{20.8}{-7.1}{$r\eta$};
  \vctxt{18.2}{-5.8}{$r^2\beta$};
  \markpts{4.4}{0}{-2}{0};
  \markpts{4.4}{0}{2}{-2};
  \markptq{4.4}{0}{0}{2};
  \markptq{4.4}{0}{0}{-2};
  \markptq{4.4}{0}{2}{0};
  \markpts{4.4}{0}{-2}{2};
 \hexLambda{0}{1.1}{-2.2};
  \hctxt{1.1}{-2.2}{$\Lambda$};
  \vctxt{-0.8}{-5.4}{$\gamma$};
  \vctxt{0.8}{-8.4}{$r\beta$};
  \vctxt{3.4}{-9.5}{$r^5\epsilon$};
  \vctxt{5.2}{-8.2}{$r^3\alpha$};
  \vctxt{3.3}{-4.9}{$r\theta$};
  \vctxt{1}{-3.75}{$r^2\beta$};
  \markpts{1.1}{-2.2}{-2}{0};
  \markptp{1.1}{-2.2}{2}{-2};
  \markptq{1.1}{-2.2}{0}{2};
  \markptq{1.1}{-2.2}{0}{-2};
  \markpts{1.1}{-2.2}{2}{0};
  \markpts{1.1}{-2.2}{-2}{2};
 \hexXi{0}{3.3}{-2.2};
  \hctxt{3.3}{-2.2}{$\Xi$};
  \vctxt{7.8}{-9.5}{$r^3\alpha$};
  \vctxt{9.4}{-12.6}{$r\epsilon$};
  \vctxt{11.9}{-13.8}{$r^2\theta$};
  \vctxt{14.0}{-12.6}{$r^3\beta$};
  \vctxt{12.1}{-9.4}{$r\eta$};
  \vctxt{9.5}{-8.0}{$r^2\beta$};
  \markptp{3.3}{-2.2}{-2}{0};
  \markpts{3.3}{-2.2}{2}{-2};
  \markptq{3.3}{-2.2}{0}{2};
  \markptq{3.3}{-2.2}{0}{-2};
  \markptq{3.3}{-2.2}{2}{0};
  \markpts{3.3}{-2.2}{-2}{2};
 \hexPi{0}{5.5}{-2.2};
  \hctxt{5.5}{-2.2}{$\Pi$};
  \vctxt{16.6}{-13.9}{$r^3\alpha$};
  \vctxt{18.4}{-16.95}{$r\epsilon$};
  \vctxt{21.1}{-18.2}{$r^5\epsilon$};
  \vctxt{22.8}{-17.0}{$r^3\alpha$};
  \vctxt{20.9}{-13.7}{$r\theta$};
  \vctxt{18.5}{-12.5}{$r^2\beta$};
  \markptp{5.5}{-2.2}{-2}{0};
  \markptp{5.5}{-2.2}{2}{-2};
  \markptq{5.5}{-2.2}{0}{2};
  \markptq{5.5}{-2.2}{0}{-2};
  \markpts{5.5}{-2.2}{2}{0};
  \markpts{5.5}{-2.2}{-2}{2};
 \hexSigma{0}{2.2}{-4.4};
  \hctxt{2.2}{-4.4}{$\Sigma$};
  \vctxt{-0.8}{-12.0}{$\zeta$};
  \vctxt{0.8}{-15.1}{$r\beta$};
  \vctxt{3.2}{-16.2}{$r^5\epsilon$};
  \vctxt{5.2}{-14.8}{$r^3\alpha$};
  \vctxt{3.1}{-11.5}{$r\gamma$};
  \vctxt{1}{-10.2}{$r^5\delta$};
  \markpts{2.2}{-4.4}{-2}{0};
  \markptp{2.2}{-4.4}{2}{-2};
  \markpts{2.2}{-4.4}{0}{2};
  \markptq{2.2}{-4.4}{0}{-2};
  \markpts{2.2}{-4.4}{2}{0};
  \markpts{2.2}{-4.4}{-2}{2};
 \hexPhi{0}{4.4}{-4.4};
  \hctxt{4.4}{-4.4}{$\Phi$};
  \vctxt{8.0}{-16.4}{$\gamma$};
  \vctxt{9.6}{-19.5}{$r\beta$};
  \vctxt{11.9}{-20.5}{$r^5\epsilon$};
  \vctxt{14.0}{-19.2}{$r^3\epsilon$};
  \vctxt{11.9}{-16.0}{$r\eta$};
  \vctxt{9.4}{-14.6}{$r^2\beta$};
  \markpts{4.4}{-4.4}{-2}{0};
  \markptp{4.4}{-4.4}{2}{-2};
  \markptq{4.4}{-4.4}{0}{2};
  \markptq{4.4}{-4.4}{0}{-2};
  \markptq{4.4}{-4.4}{2}{0};
  \markpts{4.4}{-4.4}{-2}{2};
 \hexPsi{0}{6.6}{-4.4};
  \hctxt{6.6}{-4.4}{$\Psi$};
  \vctxt{16.5}{-20.5}{$r^3\alpha$};
  \vctxt{18.2}{-23.5}{$r\epsilon$};
  \vctxt{21.0}{-24.8}{$r^5\epsilon$};
  \vctxt{22.7}{-23.6}{$r^3\epsilon$};
  \vctxt{20.8}{-20.5}{$r\eta$};
  \vctxt{18.2}{-19.2}{$r^2\beta$};
  \markptp{6.6}{-4.4}{-2}{0};
  \markptp{6.6}{-4.4}{2}{-2};
  \markptq{6.6}{-4.4}{0}{2};
  \markptq{6.6}{-4.4}{0}{-2};
  \markptq{6.6}{-4.4}{2}{0};
  \markpts{6.6}{-4.4}{-2}{2};
\end{tikzpicture}
\end{center}
\caption{The nine combinatorial hexagons (adapted from
  \cite[Fig.~4.2]{Spectre}). Vertices of type $p$, $q$, and $s$ are
  colored in red, blue, and green, respectively. Type and orientation
  of the edges are given by their labels, the power of $r$ determining
  the orientation.}
\label{fig:combhex}
\end{figure}

Let $r$ denote counterclockwise rotation by $60$ degrees.  Each of the
$9$ displayed meta-tiles in standard orientation can be multiplied by
$r^m$ with $m\in \{ 0,1, \ldots, 5\}$. The first $7$ edge types are
directional, so $\alpha$, $r\ts\alpha$, etc.~are distinct. The
reference orientation for each edge is vertical, pointing up.  (Recall
that this is the positive real direction in our identification of
$\RR^2$ with $\CC$.) The small mark used in \cite{Spectre} to indicate
the orientation of an edge has been replaced by the power of $r$
needed to map the standard edge to the given one. The last edge type,
$\eta$, is non-directional.  It is invariant under rotation by $180$
degrees, and occurs only in three rotational variants. As a chain,
$r^3 \eta = -\eta$.  The vertices $p$ and $q$ have $3$-fold symmetry,
so $r^2p=p$ and $r^2q=q$.  The third vertex, $s$, does not have any
rotational symmetry, so $s, rs, \ldots, r^5s$ are distinct.

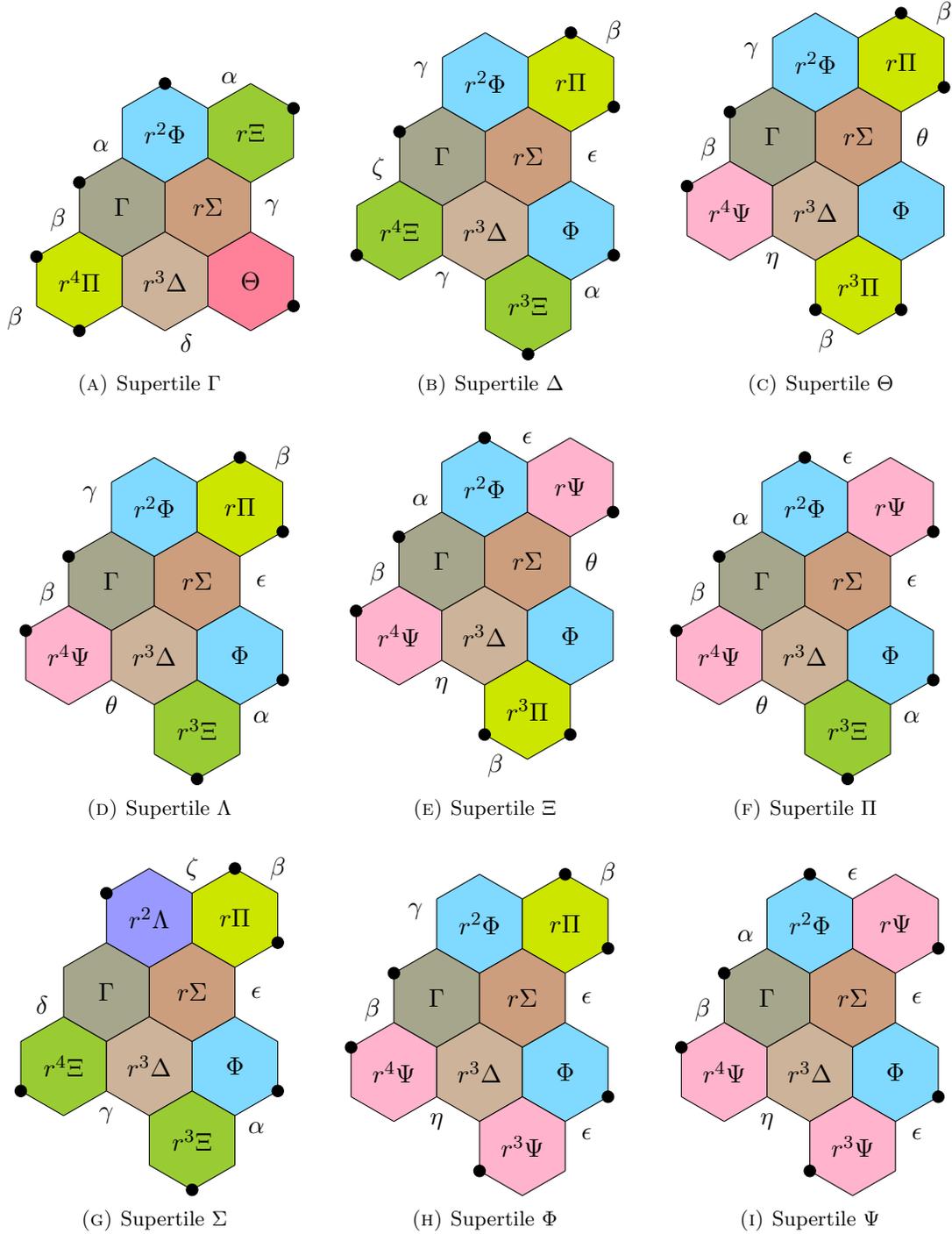
\begin{figure}[p]
\captionsetup{margin=0pt}%
\begin{center}
\subfloat[Supertile $\Gamma$]{%
\begin{tikzpicture}[x=3.75mm,y=3.75mm]
  \hexGamma{0}{0}{0};    \hctxt{0}{0}{$\Gamma$};
  \hexSigma{60}{1}{0};   \hctxt{1}{0}{$r\Sigma$};
  \hexPhi{120}{0}{1};    \hctxt{0}{1}{$r^2\Phi$};
  \hexXi{60}{1}{1};      \hctxt{1}{1}{$r\Xi$};
  \hexPi{240}{0}{-1};    \hctxt{0}{-1}{$r^4\Pi$};
  \hexDelta{300}{1}{-1}; \hctxt{1}{-1}{$r^5\Delta$};
  \hexTheta{0}{2}{-1};   \hctxt{2}{-1}{$\Theta$};
  \markpt{4}{4};
  \markpt{0}{2};
  \markpt{-2}{0};
  \markpt{0}{-4};
  \markpt{10}{-8};
  \markpt{10}{0};
  \vctxt{1}{3}{$\alpha$};
  \vctxt{-1}{1}{$\beta$};
  \vctxt{-3}{-2}{$\beta$};
  \vctxt{5}{-7}{$\delta$};
  \vctxt{9}{-3.5}{$\gamma$};
  \vctxt{7}{2.75}{$\alpha$};
\end{tikzpicture}%
} \qquad \subfloat[Supertile $\Delta$]{%
\begin{tikzpicture}[x=3.75mm,y=3.75mm]
  \hexGamma{0}{0}{0};    \hctxt{0}{0}{$\Gamma$};
  \hexSigma{60}{1}{0};   \hctxt{1}{0}{$r\Sigma$};
  \hexPhi{120}{0}{1};    \hctxt{0}{1}{$r^2\Phi$};
  \hexPi{60}{1}{1};      \hctxt{1}{1}{$r\Pi$};
  \hexXi{240}{0}{-1};    \hctxt{0}{-1}{$r^4\Xi$};
  \hexDelta{300}{1}{-1}; \hctxt{1}{-1}{$r^5\Delta$};
  \hexPhi{0}{2}{-1};     \hctxt{2}{-1}{$\Phi$};
  \hexXi{300}{2}{-2};    \hctxt{2}{-2}{$r^5\Xi$};
  \markpt{8}{2};
  \markpt{0}{2};
  \markpt{-2}{-2};
  \markpt{6}{-10};
  \markpt{10}{-8};
  \markpt{10}{-2};
  \vctxt{1}{4}{$\gamma$};
  \vctxt{-1}{1}{$\zeta$};
  \vctxt{2}{-5}{$\gamma$};
  \vctxt{9}{-9}{$\alpha$};
  \vctxt{9}{-3.5}{$\epsilon$};
  \vctxt{10}{1}{$\beta$};
\end{tikzpicture}%
} \qquad \subfloat[Supertile $\Theta$]{%
\begin{tikzpicture}[x=3.75mm,y=3.75mm]
  \hexGamma{0}{0}{0};    \hctxt{0}{0}{$\Gamma$};
  \hexSigma{60}{1}{0};   \hctxt{1}{0}{$r\Sigma$};
  \hexPhi{120}{0}{1};    \hctxt{0}{1}{$r^2\Phi$};
  \hexPi{60}{1}{1};      \hctxt{1}{1}{$r\Pi$};
  \hexPsi{240}{0}{-1};   \hctxt{0}{-1}{$r^4\Psi$};
  \hexDelta{300}{1}{-1}; \hctxt{1}{-1}{$r^5\Delta$};
  \hexPhi{0}{2}{-1};     \hctxt{2}{-1}{$\Phi$};
  \hexPi{300}{2}{-2};    \hctxt{2}{-2}{$r^5\Pi$};
  \markpt{8}{2};
  \markpt{0}{2};
  \markpt{-2}{0};
  \markpt{4}{-8};
  \markpt{8}{-10};
  \markpt{10}{-2};
  \vctxt{1}{4}{$\gamma$};
  \vctxt{-1}{1}{$\beta$};
  \vctxt{2}{-5}{$\eta$};
  \vctxt{4.5}{-9.5}{$\beta$};
  \vctxt{9}{-3.5}{$\theta$};
  \vctxt{10}{1}{$\beta$};
\end{tikzpicture}%
} \\ \subfloat[Supertile $\Lambda$]{%
\begin{tikzpicture}[x=3.75mm,y=3.75mm]
  \hexGamma{0}{0}{0};    \hctxt{0}{0}{$\Gamma$};
  \hexSigma{60}{1}{0};   \hctxt{1}{0}{$r\Sigma$};
  \hexPhi{120}{0}{1};    \hctxt{0}{1}{$r^2\Phi$};
  \hexPi{60}{1}{1};      \hctxt{1}{1}{$r\Pi$};
  \hexPsi{240}{0}{-1};   \hctxt{0}{-1}{$r^4\Psi$};
  \hexDelta{300}{1}{-1}; \hctxt{1}{-1}{$r^5\Delta$};
  \hexPhi{0}{2}{-1};     \hctxt{2}{-1}{$\Phi$};
  \hexXi{300}{2}{-2};    \hctxt{2}{-2}{$r^5\Xi$};
  \markpt{8}{2};
  \markpt{0}{2};
  \markpt{-2}{0};
  \markpt{6}{-10};
  \markpt{10}{-8};
  \markpt{10}{-2};
  \vctxt{1}{4}{$\gamma$};
  \vctxt{-1}{1}{$\beta$};
  \vctxt{2}{-5}{$\theta$};
  \vctxt{9}{-9}{$\alpha$};
  \vctxt{9}{-3.5}{$\epsilon$};
  \vctxt{10}{1}{$\beta$};
\end{tikzpicture}%
} \qquad \subfloat[Supertile $\Xi$]{%
\begin{tikzpicture}[x=3.75mm,y=3.75mm]
  \hexGamma{0}{0}{0};    \hctxt{0}{0}{$\Gamma$};
  \hexSigma{60}{1}{0};   \hctxt{1}{0}{$r\Sigma$};
  \hexPhi{120}{0}{1};    \hctxt{0}{1}{$r^2\Phi$};
  \hexPsi{60}{1}{1};     \hctxt{1}{1}{$r\Psi$};
  \hexPsi{240}{0}{-1};   \hctxt{0}{-1}{$r^4\Psi$};
  \hexDelta{300}{1}{-1}; \hctxt{1}{-1}{$r^5\Delta$};
  \hexPhi{0}{2}{-1};     \hctxt{2}{-1}{$\Phi$};
  \hexPi{300}{2}{-2};    \hctxt{2}{-2}{$r^5\Pi$};
  \markpt{4}{4};
  \markpt{0}{2};
  \markpt{-2}{0};
  \markpt{4}{-8};
  \markpt{8}{-10};
  \markpt{10}{-2};
  \vctxt{1}{3}{$\alpha$};
  \vctxt{-1}{1}{$\beta$};
  \vctxt{2}{-5}{$\eta$};
  \vctxt{4.5}{-9.5}{$\beta$};
  \vctxt{9}{-3.5}{$\theta$};
  \vctxt{6}{3}{$\epsilon$};
\end{tikzpicture}%
} \qquad \subfloat[Supertile $\Pi$]{%
\begin{tikzpicture}[x=3.75mm,y=3.75mm]
  \hexGamma{0}{0}{0};    \hctxt{0}{0}{$\Gamma$};
  \hexSigma{60}{1}{0};   \hctxt{1}{0}{$r\Sigma$};
  \hexPhi{120}{0}{1};    \hctxt{0}{1}{$r^2\Phi$};
  \hexPsi{60}{1}{1};     \hctxt{1}{1}{$r\Psi$};
  \hexPsi{240}{0}{-1};   \hctxt{0}{-1}{$r^4\Psi$};
  \hexDelta{300}{1}{-1}; \hctxt{1}{-1}{$r^5\Delta$};
  \hexPhi{0}{2}{-1};     \hctxt{2}{-1}{$\Phi$};
  \hexXi{300}{2}{-2};    \hctxt{2}{-2}{$r^5\Xi$};
  \markpt{4}{4};
  \markpt{0}{2};
  \markpt{-2}{0};
  \markpt{6}{-10};
  \markpt{10}{-8};
  \markpt{10}{-2};
  \vctxt{1}{3}{$\alpha$};
  \vctxt{-1}{1}{$\beta$};
  \vctxt{2}{-5}{$\theta$};
  \vctxt{9}{-9}{$\alpha$};
  \vctxt{9}{-3.5}{$\epsilon$};
  \vctxt{6}{3}{$\epsilon$};
\end{tikzpicture}%
} \\ \subfloat[Supertile $\Sigma$]{%
\begin{tikzpicture}[x=3.75mm,y=3.75mm]
  \hexGamma{0}{0}{0};    \hctxt{0}{0}{$\Gamma$};
  \hexSigma{60}{1}{0};   \hctxt{1}{0}{$r\Sigma$};
  \hexLambda{120}{0}{1}; \hctxt{0}{1}{$r^2\Lambda$};
  \hexPi{60}{1}{1};      \hctxt{1}{1}{$r\Pi$};
  \hexXi{240}{0}{-1};    \hctxt{0}{-1}{$r^4\Xi$};
  \hexDelta{300}{1}{-1}; \hctxt{1}{-1}{$r^5\Delta$};
  \hexPhi{0}{2}{-1};     \hctxt{2}{-1}{$\Phi$};
  \hexXi{300}{2}{-2};    \hctxt{2}{-2}{$r^5\Xi$};
  \markpt{8}{2};
  \markpt{2}{4};
  \markpt{-2}{-2};
  \markpt{6}{-10};
  \markpt{10}{-8};
  \markpt{10}{-2};
  \vctxt{6}{3}{$\zeta$};
  \vctxt{-1}{1}{$\delta$};
  \vctxt{2}{-5}{$\gamma$};
  \vctxt{9}{-9}{$\alpha$};
  \vctxt{9}{-3.5}{$\epsilon$};
  \vctxt{10}{1}{$\beta$};
\end{tikzpicture}%
} \qquad \subfloat[Supertile $\Phi$]{%
\begin{tikzpicture}[x=3.75mm,y=3.75mm]
  \hexGamma{0}{0}{0};    \hctxt{0}{0}{$\Gamma$};
  \hexSigma{60}{1}{0};   \hctxt{1}{0}{$r\Sigma$};
  \hexPhi{120}{0}{1};    \hctxt{0}{1}{$r^2\Phi$};
  \hexPi{60}{1}{1};      \hctxt{1}{1}{$r\Pi$};
  \hexPsi{240}{0}{-1};   \hctxt{0}{-1}{$r^4\Psi$};
  \hexDelta{300}{1}{-1}; \hctxt{1}{-1}{$r^5\Delta$};
  \hexPhi{0}{2}{-1};     \hctxt{2}{-1}{$\Phi$};
  \hexPsi{300}{2}{-2};   \hctxt{2}{-2}{$r^5\Psi$};
  \markpt{8}{2};
  \markpt{0}{2};
  \markpt{-2}{0};
  \markpt{4}{-8};
  \markpt{10}{-8};
  \markpt{10}{-2};
  \vctxt{1}{4}{$\gamma$};
  \vctxt{-1}{1}{$\beta$};
  \vctxt{2}{-5}{$\eta$};
  \vctxt{9}{-9}{$\epsilon$};
  \vctxt{9}{-3.5}{$\epsilon$};
  \vctxt{10}{1}{$\beta$};
\end{tikzpicture}%
} \qquad \subfloat[Supertile $\Psi$]{%
\begin{tikzpicture}[x=3.75mm,y=3.75mm]
  \hexGamma{0}{0}{0};    \hctxt{0}{0}{$\Gamma$};
  \hexSigma{60}{1}{0};   \hctxt{1}{0}{$r\Sigma$};
  \hexPhi{120}{0}{1};    \hctxt{0}{1}{$r^2\Phi$};
  \hexPsi{60}{1}{1};     \hctxt{1}{1}{$r\Psi$};
  \hexPsi{240}{0}{-1};   \hctxt{0}{-1}{$r^4\Psi$};
  \hexDelta{300}{1}{-1}; \hctxt{1}{-1}{$r^5\Delta$};
  \hexPhi{0}{2}{-1};     \hctxt{2}{-1}{$\Phi$};
  \hexPsi{300}{2}{-2};   \hctxt{2}{-2}{$r^5\Psi$};
  \markpt{4}{4};
  \markpt{0}{2};
  \markpt{-2}{0};
  \markpt{4}{-8};
  \markpt{10}{-8};
  \markpt{10}{-2};
  \vctxt{1}{3}{$\alpha$};
  \vctxt{-1}{1}{$\beta$};
  \vctxt{2}{-5}{$\eta$};
  \vctxt{9}{-9}{$\epsilon$};
  \vctxt{9}{-3.5}{$\epsilon$};
  \vctxt{6}{3}{$\epsilon$};
\end{tikzpicture}%
}%
\end{center}
\caption{Nine supertiles (adapted from \cite[Figure~5.1]{Spectre}).
  Each is drawn in the reverse handedness of the corresponding marked
  hexagon of Figure~\ref{fig:combhex}, preserving the handedness of
  the marked hexagons within it.}
\label{fig:inflhex}
\end{figure}

The substitution in terms of hexagons is shown in
Figure~\ref{fig:inflhex}, adapted from \cite[Fig.~5.1]{Spectre}.
Specifically, it shows the action of $\sigma^*$, which turns reflected
tiles into ordinary tiles.  The nine pictures show what happens when
you take the nine meta-tiles in standard orientation, reflect them
(horizontally) and then substitute. We now have all prerequisites
  to state and prove the following result. We employ a strategy where
  all arguments can be checked by hand, so do not need any
  computer-assisted steps.

\begin{theorem}\label{thm:cohomology}
  Let\/ $\Omega$ be the orbit closure of a Spectre $($or
  Hat-Turtle$\, )$ tiling. The first complex \v{C}ech cohomology of\/
  $\Omega$ is
\[
    \check{H}^1(\Omega, \CC) \, = \, \CC^4.
\]
Furthermore, this\/ $\CC^4$ decomposes as a\/ $\CC^2$ from the
fundamental\/ $r=\xi$ representation of the cyclic group\/ $C_6$ and
a\/ $\CC^2$ from the representation\/ $r=\xi^5$.  The eigenvalues of
the $($squared$\, )$ substitution on each\/ $\CC^2$ are\/
$4 \pm \sqrt{15}$, each with multiplicity\/ $1$.

The second complex \v{C}ech cohomology of\/ $\Omega$ is
\[
    \check{H}^2(\Omega, \CC) \, = \, \CC^{10},
\]
with a contribution\/ $\CC^2$ from each\/
$r \in \{ 1, -1, \xi, \xi^5 \}$ and a single\/ $\CC$ from\/ $r=\xi^2$
and $r=\xi^4$.
\end{theorem}

\begin{proof}
  By the results of \cite[Thms.~6.1 and 6.3]{AP}, the \v{C}ech
  cohomology of a substitution tiling space is the direct limit of the
  (ordinary) cohomology of the Anderson--Putnam (AP) complex $\GAP$
  under substitution. We will first compute the cohomology of $\GAP$
  and then take the direct limit. Since we are working over the
  complex numbers, this is the same as the direct sum of the
  eigenspaces of $H^*(\GAP, \CC)$ under substitution with non-zero
  eigenvalues.

  The AP complex is built from one copy of each tile type via
  identifying edges (and therefore vertices) where two tiles can meet.
  This information is already encoded in Figure~\ref{fig:combhex}.
  Including rotations (as we must), there are $54$ faces ($9$ shapes
  in $6$ orientations), $45$ edges ($7$ in $6$ orientations and $\eta$
  in $3$ orientations), and $10$ vertices ($p$ and $q$ in $2$
  orientations and $s$ in $6$). 

  Fortunately, the (co)boundary maps on $\GAP$ commute with rotation,
  so instead of using $10 {\times} 45$ or $45 {\times} 54$ matrices,
  we can use $3{\times} 8$ and $8{\times} 9$ matrices with entries
  that are polynomials in $r$. See \cite[Sec.~4]{ORS} for a
    discussion of this procedure, and of working with one
    representation of the rotation group $C_6$ at a time.  The
  boundary maps can be read off from Figure~\ref{fig:combhex}. The
  boundary map $\partial^{}_1$ on edges is given by the matrix
\begin{equation}\label{eq:partial1}
  \begin{pmatrix} 1 & 0 & 0 & 0 & 1 & 0 & 0 & 0 \\ 
    0 & -r & 0 & 0 & -r & 0 & 1 & 1\nts {-}r \\ 
    -r^3 & r^4 & 1\nts {-}r^5 & r^2{-}r^4 & 0 & r^4{-}r^5
      & -r^5 & 0 \end{pmatrix},
\end{equation}
while the boundary map $\partial^{}_2$ of faces is given by the matrix
\begin{equation}\label{eq:partial2}
 \begin{pmatrix}
   r^3{-}r & -r^3 & 0 & -r^3 & r^3 & 0 & -r^3 & 0 & r^3 \\ 
   r^2{-}r^4 & -r & -r{+}r^2{-}r^3 & r^2{-}r
       & r^2{-}r^3 & r^2 & -r& r^2{-}r & r^2 \\ 
   r^5 & r{-}1 & -1 & -1 & 0 & 0 & r & -1 & 0 \\ 
   1 & 0 & 0 & 0 & 0 & 0 & -r^5 & 0 & 0 \\ 
   0 & r^5 & 0 & r^5 & -r & r^5{-}r & r^5
       & r^5{-}r^3 & -r{-}r^3{+}r^5 \\ 
   0 & r^2 & 0 & 0 & 0 & 0 & -1 & 0 & 0 \\ 
   0 & 0 & -r^2 & r & -r^2 & r & 0 & 0& 0  \\ 
   0 & 0 & r & 0 & r & 0 & 0 & r& r 
\end{pmatrix}
\end{equation}
 
The matrices $M_1^*$ and $M_2^*$ for the substitution on edges and
faces can similarly be read off from Figure~\ref{fig:inflhex}. We get
\[
  M^{*}_{1} \, = \,
  \begin{pmatrix} 
   0 & -r^5 & 0 & r^2 & 0 &  -r^5 & 0 & 0 \\ 
   -r^5 & 0 & r^2 & r & 0 & 0 & 0 & 0 \\ 
   0 & 0 & 0 & 0 & 0 & 0 & 0 & 0 \\ 
   0 & 0 & 0 & 0 & 0 & 0 & 0 & 0 \\ 
   r^2 & r & r{-}r^5 & -r^4 & r {+} r^2 \nts {-} r^5
      & r & r {+} r^2 \nts {-} r^4 {-} r^5
      & r {+} r^2\nts {-} r^4 {-} r^5 \\ 
   0 & 0 & 0 & 0 & 0 & 0 & 0 & 0 \\ 
   0 & 0 & r^2 & r^2\nts {-}r^5 & 0 & r^2 & r^2 & 0 \\ 
   r^3 & 0 & -1 & 0 & -1 & 0 & -1 & r^3 
\end{pmatrix},
\]
and
\[
  M^{*}_{2} \, = \,
  \begin{pmatrix} 
   1&1& 1 & 1 & 1 & 1 & 1 & 1 & 1 \\ 
   r^5 & r^5 & r^5 & r^5 & r^5 & r^5 & r^5 & r^5 & r^5 \\ 
   1 & 0 & 0 & 0 & 0 & 0 & 0 & 0 & 0 \\ 
   0 & 0 & 0 & 0 & 0 & 0 & r^2 & 0 & 0 \\ 
   r  & r^4{+}r^5 & 0 & r^5 & 0 & r^5 & r^4{+}r^5 & 0 & 0 \\ 
   r^4 & r & r{+}r^5 & r & r^5 & 0 & r & r & 0 \\ 
   r & r & r & r & r & r & r & r & r \\ 
   r^2 & 1\nts {+} r^2 & 1\nts {+} r^2 & 1\nts {+} r^2
       & 1\nts {+} r^2 & 1\nts {+} r^2 & 1
       & 1\nts {+} r^2 & 1\nts {+} r^2 \\ 
   0 & 0 & r^4 & r^4 & r{+}r^4 & r{+}r^4 & 0
       & r^4{+}r^5 & r{+}r^4{+}r^5 
\end{pmatrix}.
\]
The matrices describing two rounds of substitution, taking unreflected
meta-tiles to unreflected meta-tiles, are then $M_1^*M^{}_1$ and
$M_2^* M^{}_2$, with the details as follows.

As  in \cite[Sec.~2]{BGS}, we work with one representation of
$C_6$ at a time. That is, we replace $r$ with a power of $\xi$ in
\eqref{eq:partial1} and \eqref{eq:partial2} and compare kernels and
images of $\partial^{}_1$ and $\partial^{}_2$ and eigenspaces of
$M_1^* M^{}_1$ and $M_2^* M^{}_2$. Since we are computing cohomology
rather than homology, we are acting by $\partial^{}_1$ and
$\partial^{}_2$ on the right and our kernels and images will be sets
of row vectors, not column vectors. Since $r^3 \eta = -\eta$, the last
column of $\partial^{}_1$ and $M_1^*$ and the last row of
$\partial^{}_2$ and $M_1^*$ should be deleted when $r^3 \ne -1$, or in
other words, when $r=1$, $\xi^2$, or $\xi^4$. Since $r^2p=p$ and
$r^2q=q$, the first two rows of $\partial^{}_1$ should be deleted when
$r^2 \ne 1$, that is, when $r=\xi$, $\xi^2$, $\xi^4$ or $\xi^5$. Since
reflection turns $r$ into $r^{-1}$, and since the inverse of a unit
complex number is its complex conjugate, the matrices $M_1$ and $M_2$
in each representation are just the complex conjugates of the matrices
$M_1^*$ and $M_2^*$.

We summarize the calculations and results, one representation at a
time.

\begin{itemize}
\item[$r=1$.] When $r=1$, we have $3$ vertices, $7$ edges and $9$
  faces. The matrix $\partial^{}_1$ has rank $2$, while
  $\partial^{}_2$ has rank $5$.  Since $2+5=7$, this representation
  contributes nothing to $\check{H}^1$.  Since $9-5=4$, this
  representation contributes $\CC^4$ to $H^2(\GAP, \CC)$. Picking
  generators of this $4$-dimensional space, we see how these generators
  transform under $M_2^* M_2$. The $4 {\times} 4$ matrix that
  describes this action has rank $2$ and non-zero eigenvalues
  $(4 \pm \sqrt{15} \, )^2$. This representation thus only contributes
  $\CC^2$ to $\check{H}^2(\Omega, \CC)$. The generators of
  $\check{H}^2(\Omega,\CC)$ can in fact be represented by cochains
  that count Spectres and Mystics without distinguishing between the
  different kinds of collared Spectres.

\item[$r=\xi$.] When $r=\xi$, we have one vertex, $8$ edges and $9$
  faces.  The rank of $\partial^{}_1$ is $1$ while that of
  $\partial^{}_2$ is $5$.  This representation then contributes
  $\CC^2$ to $H^1(\GAP, \CC)$ and $\CC^4$ to $H^2(\GAP, \CC)$. Under
  substitution, the first cohomology transforms with eigenvalues
  $4 \pm \sqrt{15}$, while the second cohomology transforms with
  eigenvalues $4 \pm \sqrt{15}$ and $0$ (the latter twice). That is,
  all of the first cohomology survives to the direct limit,
  contributing $\CC^2$ to $\check{H}^1(\Omega, \CC)$, but only a
  $\CC^2$ subspace of $H^2(\GAP, \CC)$ survives to the direct limit.
\item[$r=\xi^2$.] When $r=\xi^2$, we have one vertex, $7$ edges and
  $9$ faces. The rank of $\partial^{}_1$ is $1$ and the rank of
  $\partial^{}_2$ is $6$, so we get no first cohomology and a
  contribution of $\CC^3$ to $H^2(\Gamma_{_\mathrm{\!
      AP}},\CC)$. However, substitution on the cokernel of
  $\partial^{}_2$ involves a $3{\times} 3$ matrix of rank $1$, whose
  only non-zero eigenvalue is $1$. Consequently, only a single 
    copy of $\CC$ survives the direct limit to contribute to
  $\check{H}^2(\Omega, \CC)$.

\item[$r=-1$.] When $r=-1$, we have $3$ vertices, $8$ edges and $9$
  faces. The ranks of $\partial^{}_1$ and $\partial^{}_2$ are $3$ and
  $5$, respectively. This representation contributes nothing to
  $H^1(\GAP,\CC)$ and $\CC^4$ to $H^2(\GAP,\CC)$. The $4 {\times} 4$
  matrix describing substitution on this $\CC^4$ has rank $2$, with a
  double non-zero eigenvalue $1$. The direct limit of $\CC^4$ under
  substitution thus is $\CC^2$.

\item[$r=\xi^4$.] The calculations for $r=\xi^4$ are identical (up to
  complex conjugation) to those for $r=\xi^2$.

\item[$r=\xi^5$.] The calculations for $r=\xi^5$ are identical (up to
  complex conjugation) to those for $r=\xi$. \qedhere
\end{itemize}
\end{proof}

\begin{remark} 
  $H^2(\GAP,\CC)=\CC^{22}$ is much bigger than
  $\check{H}^2(\Omega, \CC)=\CC^{10}$. In every single representation,
  two of the eigenvalues of substitution on $H^2(\GAP,\CC)$ are
  zero. As a result, $2 \cdot 6=12$ copies of $\CC$ fail to
  survive to the direct limit.  This strongly suggests that our
  description of the substitutive structure of the Spectre tiling
  space is overly complicated. A different collaring scheme, or a
  different set of basic shapes that get substituted, would yield a
  different AP complex and could plausibly yield a simpler cohomology
  calculation. Since the publication of the original approach
    \cite{Spectre}, alternate substitutive schemes have been proposed
  \cite{Smith-J}.  It will be interesting to see how much the
  calculations can be streamlined by using them. \exend
\end{remark}

\begin{remark}
  The decomposition of the chain and cochain complexes of $\GAP$ is
  justified when working over $\CC$, but not when computing the
  integral cohomology $\check{H}^*(\Omega, \ZZ)$. To compute the
  integral cohomology, we must replace each entry of $\partial_1$,
  $\partial_2$, $M_1$ and $M_2$ with a $6 {\times} 6$ matrix (with
  adjustments for the rows and columns that refer to $\eta$, $p$ and
  $q$), replacing the number 1 with the identity matrix and replacing
  $r$ with the cyclic permutation matrix
\[
    \left ( \begin{smallmatrix} 0&0&0&0&0&1 \cr 1&0&0&0&0&0 \cr
      0&1&0&0&0&0 \cr 0&0&1&0&0&0 \cr 0&0&0&1&0&0 \cr 0&0&0&0&1&0
    \end{smallmatrix} \right ).
\]
Comparing the (left-)kernels and images of these matrices shows that
$\check{H}^1(\Omega, \ZZ)=\ZZ^4$ and that
$\check{H}^2(\Omega, \ZZ) = \ZZ^{10}$. \exend
\end{remark}

We recall some facts about shape changes from \cite{CS1}.  Shape
changes to a $d$-dimensional tiling are classified, up to MLD
equivalence, by $\check H^1(\Omega, \RR^d)$; see
\cite[Thm.~2.1]{CS1}. Since $d=2$, we use $\CC$ in place of
$\RR^2$. For a substitution tiling, the contracting subspace of
$\check H^1(\Omega, \RR^d)$ under substitution parametrizes shape
changes that are topological conjugacies \cite[Thm.~2.2]{CS1}. In our
calculation of $\check H^1(\Omega, \CC)$, this is the eigenspace of
substitution with eigenvalue $4-\sqrt{15}$. With these facts in hand,
we now analyze shape changes for the Spectre tilings.

\begin{theorem}\label{thm:symmetric}
  Any shape change obtained by changing the values of\/ $a$ and\/ $b$
  is the composition of a rescaling, a rotation, and a topological
  conjugacy.
\end{theorem}

\begin{proof}
  Changing $a$ and $b$ preserves $6$-fold rotational symmetry. This
  means that the shape class of each tiling lies in the portion of
  $\check{H}^1(\Omega, \CC)$ that comes from the fundamental
  representation $r=\xi$. It is easy to check that changes in the
  parameters $a$ and $b$ correspond to linearly independent elements
  of $\check{H}^1(\Omega, \CC)$, so \emph{all} rotationally symmetric
  shape changes are MLD to changes in $a$ and $b$.

 The action of substitution on rotationally symmetric shape
    classes has an expanding eigenvalue $4 + \sqrt{15}$ and a
    contracting eigenvalue $4-\sqrt{15}$.  All shape classes must have
    a nonzero component in the direction of the expanding eigenvector,
    as this accounts for the asymptotic areas of supertiles, which
    scale by $(4+\sqrt{15})^2$ under substitution.  Therefore, the
    shape class for any Spectre tiling can be obtained from that of
    any other Spectre tiling by first multiplying $a$ and $b$ by a
    nonzero complex number $z$ and then adding a multiple of the
    contracting eigenvector.  The first move is the same as rescaling
    the tiling by $|z|$ and rotating by the argument of $z$. The
    second move is a topological conjugacy.

  Note that not all values of $a$ and $b$ yield actual tiles
  Tile($a,b$). The boundary of Tile($a,b$) is a sequence of edges with
  displacements $a$ and $\ii \ts b$ times powers of $\xi$. In some
  cases, this traces a figure 8 instead of a simple closed curve.  In
  such cases, we must work instead with larger clusters, equivalent to
  what is obtained by substituting Hats and Mystics several
  times. Under substitution, the direction of the shape class in
  $\check{H}^1(\Omega, \CC)$ approaches the eigenvector with
  eigenvalue $4+\sqrt{15}$, yielding shapes whose boundaries
  \emph{are} simple closed curves. The only cost of this is that we
  have to treat the $n$-supertiles, rather than the individual Hats
  and Turtles (or Mystics), as the basic units of our tiling.
\end{proof} 

We next extend our analysis to shape changes that may not respect
  rotational symmetry.

\begin{theorem}\label{thm:asymmetric}
  Any tiling space homeomorphic to a space of Spectre tilings is
  topologically conjugate to a linear transformation applied to an
  arbitrarily chosen Spectre tiling space.
\end{theorem}

\begin{proof}
  Every homeomorphism of tiling spaces with \emph{finite local
    complexity} (FLC) is homotopic to the composition of a shape
  change and an MLD equivalence \cite[Cor.~8.3]{JS}. This
  implies that any FLC tiling space that is homeomorphic to a
  specific FLC tiling space $\Omega$ is MLD to the result of a shape
  change applied to $\Omega$. However, we have already identified all
  of the shape changes to the Spectre space. Those that preserve
  rotational symmetry are equivalent to varying $a$ and $b$, or,
  equivalently, to a combination of a topological conjugacy and a
  rescaling and rotation.

  Next, we consider shape changes from the $r=\xi^5$ representation
  that break rotational symmetry.  Among these are linear
  transformations, such as shears, that do not commute with
  rotation. Also among these are the eigenspace of substitution with
  eigenvalue $4-\sqrt{15}$. Together, these span the $\CC^2$ of
  $\check{H}^1(\Omega, \CC)$ that comes from $r=\xi^5$.

  Put another way, the entire space of shape changes has real
  dimension $8$. Four of these dimensions are linear transformations
  (two of which involve expansion and rotation, while the other two
  involve shears). The other four correspond to the $2$
  complex-dimensional eigenspace of substitution with eigenvalue
  $4-\sqrt{15}$, all of which induce topological conjugacies.
\end{proof} 

The proofs of Theorems~\ref{thm:symmetric} and \ref{thm:asymmetric}
are essentially the same as the proofs of analogous results about the
Hat tilings \cite{BGS}.  Once we know that $\check{H}^1(\Omega, \CC)$
is as small as it could possibly be, the rest follows.

\section{CASPr, the friendly Spectre}\label{sec:selfsim}

In what follows, we work exclusively with the meta-tiles, the
combinatorial hexagons from Figure~\ref{fig:combhex}, and call such
tilings Spectre tilings, too. In order to find a self-similar
representative of this family of Spectre tilings, we first determine
its edge vectors by solving an eigenvalue equation.  Specifically, the
square of the edge inflation, $M_1^* M^{}_1$, must scale the edge
vectors by a factor $\lambda = 4 + \sqrt{15}$. A solution of the
  eigenvector equation $M_1^* M^{}_1 \bs{e} = \lambda \bs{e} $, with
$\xi=\ee^{2\pi\ii/6}$ as before, is given by
\begin{equation}\label{eq:edges}
\begin{split}
     e^{}_\alpha  \, &= \,  2\ts\xi + (1-\xi) \lambda\ts , \\
     e^{}_\beta   \, &= \, 1 - 3\ts\xi + \xi \lambda\ts , \\
     e^{}_\gamma  \, &= \, -2 + 4\ts\xi + (2 - \xi) \lambda\ts , \\
     e^{}_\delta  \, &= \, -9 + 3\ts\xi + 3 \lambda\ts , \\
     e^{}_\epsilon \, &= \, 1 -\xi   + \lambda\ts , \\
     e^{}_\zeta   \, &= \, -1 - 4\xi + (1+\xi) \lambda\ts , \\
     e^{}_\theta  \, &= \, -1 + \xi  + 2 \lambda\ts , \\
     e^{}_\eta    \, &= \, 1 + 2\ts\xi + \lambda\ts . 
\end{split}
\end{equation}
All edge vectors are elements of the $\ZZ$-module
$\ZZ[\xi,\lambda] = \langle 1, \xi, \lambda, \lambda \xi\ts
\rangle^{}_{\ZZ}$. The edge vectors span the edge module $E$, which is
a submodule of $\ZZ[\xi,\lambda]$ of index $9$.

\begin{remark}\label{rem:max-order}
  Let us comment on the natural number-theoretic setting of our
  system.  The $\ZZ$-module $\cO = \ZZ[\xi, \lambda ]$ with
  $\xi=\ee^{2\pi\ii/6}$ and $\lambda = 4 + \sqrt{15}$ is a ring, and
  as such a non-maximal order of the quartic number field
  $K = \QQ (\sqrt{-3}, \sqrt{-5} \, )$, which (applying Hilbert's
  theorem) can also be written as $K = \QQ (\alpha)$ with
  $\alpha = \sqrt{5} \, \ee^{2 \pi \ii/12}$; see \cite{data} for 
  the main properties of $K$. 

  The maximal order of $K$ is
  $\cO^{}_{K} = \langle 1, \alpha, \alpha^2\! /5, \alpha^3\! /5
  \rangle^{}_{\ZZ}$, which contains all algebraic integers from $K$,
  and $\cO$ is a submodule of index $3$. In fact, keeping track of the
  indices, one has
\[
  \ii \sqrt{3}\, \cO^{}_{K} \, \overset{3}{\subset} \,
  \cO \, \overset{3}{\subset} \, \cO^{}_{K} \ts ,
\]
where $\cO$ is not an ideal in $\cO^{}_{K}$, because $1 \in \cO$.  In
particular, $\cO$ is not Dedekind --- for instance, $3\ts \cO^{}_{K}$
is an ideal both in $\cO^{}_{\nts K}$ and in $\cO$, but it is not
invertible as a fractional $\cO$-ideal. Let us also remark that $\cO$
and $\cO^{}_{\nts K}$ have the same unit group, which is
\[
  \cO^{\times} \, = \: \cO^{\times}_{K} \, = \,
  \langle \ts\xi\ts \rangle  \times \langle \lambda \xi \ts\rangle
  \, \simeq \, C_6 \times C_{\infty} \ts .
\]
This makes working with $\cO$ rather natural in our setting. In
particular, important objects such as the return module will be
invariant under the action of the unit group.

Since the field $K$ has class number $2$, not all ideals of
$\cO^{}_{K}$ are principal, and neither are those of $\cO$.
Non-principal $\cO$-ideals will appear in our context, where we then
meet the special situation that they are still generated by two
elements. This follows from \cite[Thm.~2.3]{Greither}, because both
$\cO$ and $\cO^{}_{K}$ are Gorenstein rings, and no other ring lies
between them.  \exend
\end{remark}

\begin{figure}
\begin{center}
\includegraphics[width=0.85\textwidth]{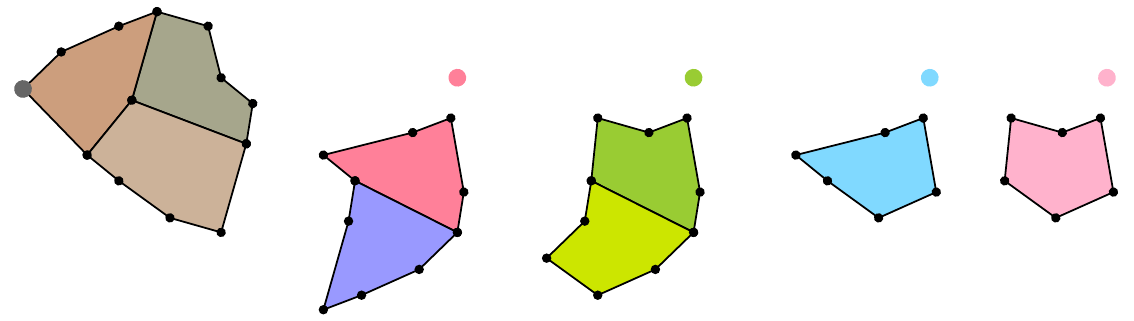}
\end{center}
\caption{\label{fig:tiles} Self-similar tiles of the CASPr
  tiling, with the edges from Eq.~\eqref{eq:edges}.
  From left to right are shown: the clusters
  $\Delta + \Gamma + \Sigma$, $\Lambda + \Theta$, and $\Pi + \Xi$, and
  the single tiles $\Phi$ and $\Psi$.  Tiles which always occur
  together (the left three clusters) are drawn together. For each of
  these (clusters of) tiles, also its control point is shown (in $5$
  colors).  }
\end{figure}

The self-similar tiles are now obtained by replacing the combinatorial
edges of the $9$ hexagon tiles from Figure~\ref{fig:combhex} by the
geometric edges from \eqref{eq:edges}, thereby producing the geometric
(non-regular) hexagon tiles shown in Figure~\ref{fig:tiles}. In the
latter, we have drawn certain clusters of tiles together, because
these tiles will always occur together in any legal tiling. By
construction, this self-similar tiling is related by a shape change to
the combinatorially equivalent tiling by regular hexagons, and hence
is related by a shape change to the original Spectre tiling.

The inflation of the geometric tiles can now be read off from
Figure~\ref{fig:inflhex}, again replacing the combinatorial edges by
the geometric ones. The inflation of the tiles and clusters from
Figure~\ref{fig:tiles} are shown in Figure~\ref{fig:inftiles}.  We
have drawn the (outer) supertile edges in different colors, depending
on the type of the supertile edge.  As one can see, if matching
supertile edges are joined, each edge type always has the same
environment on both sides, with the exception of the short $\beta$
edge (red), which on one side can either have a $\Gamma$ and a $\Xi$,
or a $\Gamma$ and a $\Psi$. However, the $\Xi$ and the $\Psi$ both
touch the $\beta$ superedge with their $\eta$ edge (green), so that
after one further inflation, also the $\beta$ superedge always has the
same environment on both sides. This proves that the inflation shown
in Figure~\ref{fig:inftiles} indeed forces the border, so that the
simplified cohomology computation with uncollared tiles is justified.
This is no surprise, of course, as the meta-tiles have been
constructed with this goal in mind.

\begin{figure}
\begin{center}
\includegraphics[width=0.6\textwidth]{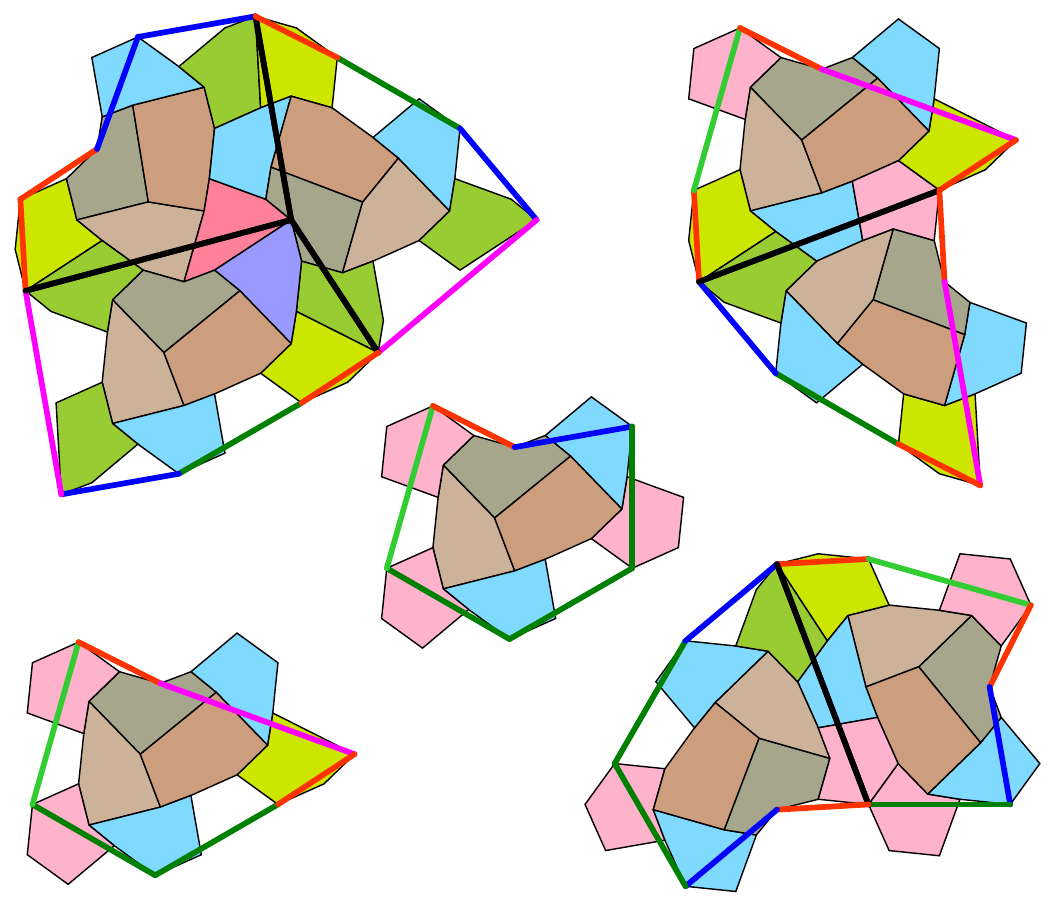}
\end{center}
\caption{\label{fig:inftiles} Inflation of the tile clusters
  $\Delta + \Gamma + \Sigma$ (top left), $\Lambda + \Theta$ (top
  right), $\Pi + \Xi$ (bottom right), $\Psi$ (centre), and $\Phi$
  (bottom left). The (outer) supertile edges are drawn in color, as a
  function of edge type: blue for $\alpha$, red for $\beta$, purple
  for $\gamma$, cyan for $\epsilon$, and green for $\eta$. The
  remaining edge types only occur in the interior of these clusters.}
\end{figure}

Let us explain one important subtlety that reflects the substitution
structure from Section~\ref{sec:top}. Obviously, the tiles in
Figure~\ref{fig:tiles} are not mirror-symmetric.  We call them
right-handed tiles. The outlines of the supertiles in
Figure~\ref{fig:inftiles} are left-handed, however. This is due to our
use of a square root of the inflation described earlier: in one
inflation step, we scale the tiling by $\sqrt{4+\sqrt{15}}$, reflect
it (thus obtaining a tiling of left-handed
supertiles), and replace each of these by a patch of right-handed
tiles as shown in Figure~\ref{fig:inftiles}. Due to the reflection,
this inflation is not a local subdivision of the scaled tiling,
but its square is.  A larger patch of tiles, generated by this
inflation, is shown in Figure~\ref{fig:tiling}, where one can see
that, indeed, the tiles $\Delta + \Gamma + \Sigma$ (drawn in three
different shades of brown) always occur together. The same holds for
the pairs $\Xi + \Psi$ (green and yellow) and $\Lambda + \Theta$ (blue
and red).  We call this self-similar Spectre tiling the \emph{friendly
  Spectre}, abbreviated as \emph{CASPr} (for
Cut-And-Symmetrically-Project).

\begin{figure}
\begin{center}
\includegraphics[width=0.95\textwidth]{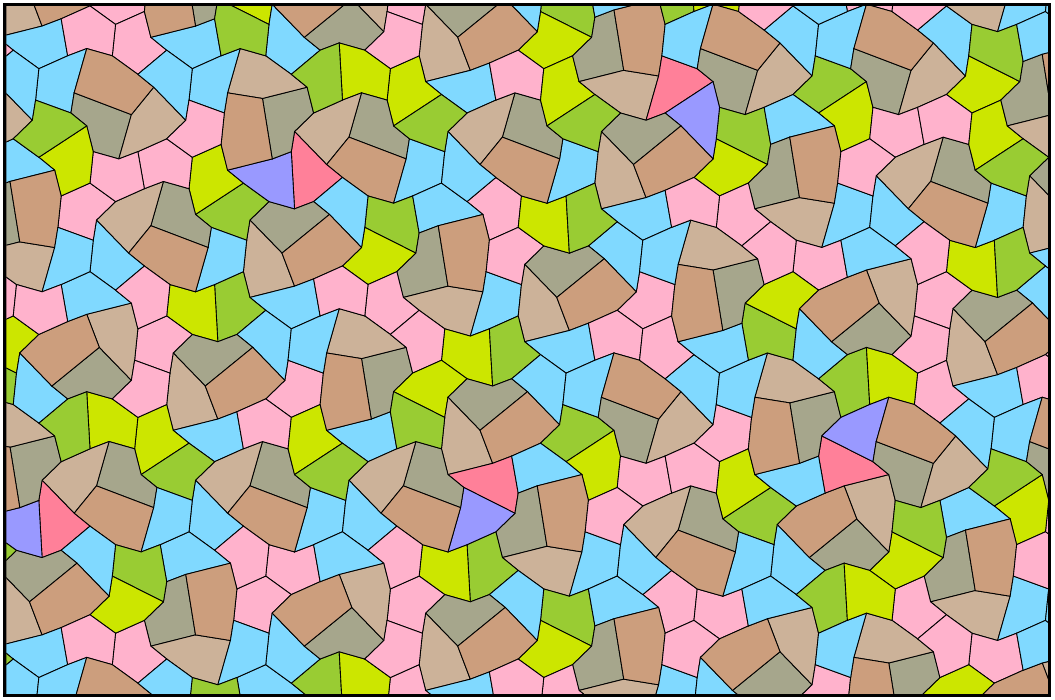}
\end{center}
\caption{\label{fig:tiling}
  Patch of a CASPr tiling.
} 
\end{figure}

Our method for taking the square root of the inflation
$\sigma^*\sigma$ deserves a closer look. The inflation $\sigma^*$ on
the left-handed tiles is conjugate to $\sigma$ via the horizontal
reflection $m$, meaning $\sigma^*= m \sigma m$, so that
$\sigma^*\sigma = (m \sigma)^2 = \widetilde{\sigma}^2$. In an
analogous way, we can write the inflations on \mbox{$1$-~and}
$2$-cohomology as
$M_1^*M^{}_1 = \widetilde{M}_1^{\, 2_{\vphantom{\chi}}}$ and
$M_2^*M^{}_2 = \widetilde{M}_2^{\, 2_{\vphantom{\chi}}}$,
respectively, in line with our approach above. Working with such
square roots is often computationally simpler than working with $M_i$
and $M_i^*$ separately. We note that the induced action of
$\widetilde{M}_1$ on the edge module $E$ leaves it invariant, whereas
$M_1$ maps it to its mirror image, which is different. Conversely,
$\widetilde{M}_1$ maps $\ZZ[\xi,\lambda]$ to a mirror image different
from $\ZZ[\xi,\lambda]$, whereas $M_1$ leaves it invariant.

Another relevant module is the return module $L$, which is the
$\ZZ$-span of all vectors translating a tile in a tiling to an
equivalent tile in the same tiling. This is a submodule of the edge
module $E$, given by the (right) kernel of $\partial^{}_1$, projected
to the edge module. The return module has index $9$ in the edge
module, and is invariant under $\widetilde{M}_1$, but not under
$M_1$. Simple bases of $E$ and $L$ in terms of the (above chosen)
standard basis of $\ZZ[\xi,\lambda]$ are given by
\begin{align}
  E \, &= \, \langle ( 1, 0, 0, 1 ), ( 0, -1, 1,-1 ),
           (-1,-1, 1, 1 ), ( 0, 1, 2, 1 ) \rangle^{}_{\ZZ}
     \quad \text{and} \\
  L \, &= \, \langle (-1,-1, 1,-2 ), ( 2, -1, 1, 1 ),
                 ( 2, 2, 1,-2 ), (-2, 1, 2, 2 ) \rangle^{}_{\ZZ} \ts .
\end{align}
In Figure~\ref{fig:tiles}, we have added special reference points,
so-called \emph{control points}, to each tile or cluster of
tiles. These have been chosen such that the set of control points of
all tiles in a tiling lie in a single translation orbit of the return
module, and such that the colored control point set is MLD with the
tiling.

\begin{remark}\label{rem:return}
  The return module $L$, viewed as a subset of $\CC$, is a
  $\ZZ$-module with generators $1+\xi+2\lambda+5\lambda\xi$,
  $3\xi+6\lambda\xi$, $3\lambda+3\lambda\xi$, and $9\lambda\xi$. As
  such, it has index $81$ in $\cO=\ZZ[\xi,\lambda]$. Since
  $\ZZ[\xi,\lambda] L \subseteq L$, it is an \emph{ideal} in $\cO$,
  thus matching the rotation and inflation symmetry of the CASPr
  tiling. However, it is not principal, as mentioned before in
    Remark~\ref{rem:max-order}.

  Via a standard lattice reduction algorithm, one finds
  $L = \langle g^{}_{1}, g^{}_{2}, g^{}_{3}, g^{}_{4} \rangle^{}_{\ZZ}$
  with generators
\begin{align*}
  g^{}_{1} \, & = \, -1-\xi+\lambda-2\lambda\xi \ts , \\
  g^{}_{2} \, & = \, 1 - 2\ts\xi + 2\lambda - \lambda\xi
                \, = \, \xi \ts g^{}_{1} \ts , \\
  g^{}_{3} \, & = \, -2+\xi+2\lambda+2\lambda\xi \ts , \\
  g^{}_{4} \, & = \, -1 - \xi- 2 \lambda + 4\lambda\xi
               \, = \, \xi \ts g^{}_{3} \ts . 
\end{align*}
With this, one can identify $L$ as the non-principal $\cO$-ideal
$L = (g^{}_{1}, g^{}_{3}) = (g^{}_{1}) + (g^{}_{3})$. Alternatively,
one can use generators in terms of $\alpha$, where one choice is
$L = \bigl(\frac{3}{5} \alpha^2 + 3 \alpha + 6, \frac{3}{5} \alpha^3 +
3 \alpha\bigr) = (3\xi + 3\alpha + 6, 3 \alpha + 3 \alpha \xi)$. We
also note that $L$ is an ideal both in $\cO$ and in $\cO^{}_{K}$, and
the given generators work for both variants. \exend
\end{remark}

\section{Embedding the Spectre}\label{sec:embed}

Equipped with the return module, we can now lift the Spectre tiling,
or rather its set of control points, to a \emph{cut-and-project
  scheme} (CPS) \cite[Sec.~7.2]{TAO} -- provided the Spectre tiling
has pure-point dynamical spectrum. The latter can be verified with the
(generalized) overlap algorithm \cite{Sol97,AL} (which we have done
prior to this step), or it can be proved in retrospect, after lifting
the tiling (see below).

An element of $\ZZ[\xi,\lambda]\subset\CC$ can be lifted to $\CC^2$ by
pairing it with its Galois conjugate (Minkowski embedding), where
$\xi \mapsto \bar{\xi}$ and $\lambda \mapsto 8 - \lambda$. This
induces a corresponding lifting of the return module
$L \subset \ZZ[\xi,\lambda]$, and results in the CPS
\begin{equation}\label{eq:CPS}
\renewcommand{\arraystretch}{1.2}\begin{array}{r@{}ccccc@{}l}
   \\  & \CC & \xleftarrow{\;\;\; \pi \;\;\; }
   & \CC \times \CC &
   \xrightarrow{\;\: \pi^{}_{\text{int}} \;\: } & \CC & \\
   & \cup & & \cup & & \cup & \hspace*{-1.5ex}
   \raisebox{1pt}{\text{\footnotesize dense}} \\
   & \pi(\cL) & \xleftarrow{\;\ts 1:1 \;\ts } & \cL &
     \xrightarrow{ \qquad } &\pi^{}_{\text{int}} (\cL) & \\
   & \| & & & & \| & \\
   & L & \multicolumn{3}{c}{\xrightarrow{\qquad\quad\quad
    \star \,\; \quad\quad\qquad}}
   &  L^{\star}  & \\ \\
\end{array}\renewcommand{\arraystretch}{1}
\end{equation}
where $\cL = \{ (x, x^{\star}) : x \in L \}$ is the Minkowski
embedding of the return module. By our choice, it agrees with the
lattice generated by the lifted control points, and we have
$L = \pi(\cL)$. Note that this simple connection between the return
module and the embedding lattice only works for self-similar tilings.
This is the reason for constructing a self-similar member of the
Spectre tiling class. To continue, we shall also need the dual lattice
$\cL^*$ (note the different star symbol), as defined with respect to
the standard inner product of $\RR^4 \simeq \CC^2$; see
Remark~\ref{rem:dual} for more.

\begin{figure}
\begin{center}
\includegraphics[width=0.8\textwidth]{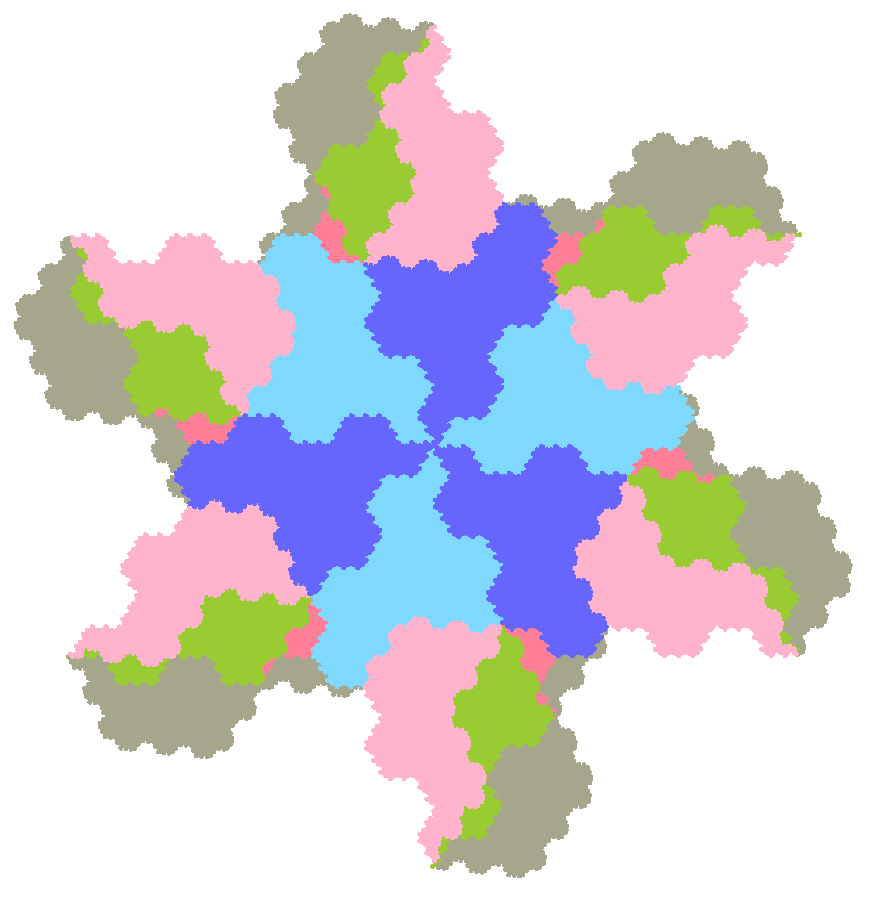}
\end{center}
\caption{\label{fig:win} Window system for the CPS of the CASPr
  tiling. The different subwindows are $5$-colored according to the
  tile (cluster) type to which the control points belong. The light
  and dark blue subwindows in the middle belong to the same tile type;
  we have chosen different colors only to distinguish distinct
  orientations. The six black dots mark the corners of a regular
    hexagon that is a fundamental domain of a lattice. The window and
    its lattice translates form a tiling of the plane with a single
    `fractile'.}
\end{figure}

If we lift all control points of an inflation fixed point tiling, and
project them to internal space, coloring them according to the tile or
cluster type to which they belong, we obtain the picture shown in
Figure~\ref{fig:win}. Here, one can distinguish different subwindows,
in which the projected control points of a given type are dense.  Each
subwindow has fractal boundaries with Hausdorff dimension
\begin{equation}\label{eq:Hausdorff}
  d^{}_{\mathrm{H}} \, = \, \myfrac{\log(5+2\sqrt{6} \,)}
  {\log(4+\sqrt{15}\, )} \, \approx \, 1.110{\ts}977 \ts ,
\end{equation}
as can be extracted from the orbit separation dimension of the CASPr
tiling; see \cite{BGG} for details.  The fractal nature appears on
different scales, but (unlike the window for the CAP tiling) the
Hausdorff dimension seems to be the same everywhere along the
boundary. This certainly deserves a closer analysis in the future.

\begin{remark}
  A more detailed approach would proceed as follows. One starts with
  the inflation rule for the five effective prototiles from
  Figure~\ref{fig:tiles}, and considers a fixed point of (a suitable
  even power of) the inflation. With respect to the control points,
  this defines the (translational) inflation displacement matrix,
  which has dimension $30$ in this case (from the six distinct
  orientations of each of the five prototiles). This gives a fixed
  point equation for the $5$-color Delone set of control points.

  This equation can then be lifted via the $\star$-map to internal
  space, where (upon taking closures) it turns into a contractive
  iterated function system on $(\cK \CC)^{30}$, where $\cK \CC$ is the
  space of non-empty compact subsets of $\CC$ equipped with the
  Hausdorff distance; see \cite[Sec.~7.1]{TAO} for a
    one-dimensional example with all details of the construction.
  Its unique attractor consists of the $30$ subwindows for the five
  types of control points, in six orientations each. When explicitly
  implementing this into an algebraic program, one can check
  \cite{Jan} via the random tracing algorithm (known as the `chaos
  game') that one indeed obtains the window system of
  Figure~\ref{fig:win}. The justification of this step relies on
  Elton's ergodic theorem \cite{Elton}; see \cite{BGMM} for further
  details.  \exend
\end{remark}

As already remarked at the beginning of this section, we have
  verified that the CASPr tiling has pure-point dynamical spectrum
  using the (generalized) overlap algorithm \cite{Sol97,AL}. As this
  is a computer-assisted proof, which is not easily repeated, we
  sketch here an alternative proof, which can be verified by hand,
  just as we did for the Hat tiling \cite{BGS}.

  In order to verify that the Spectre tiling has pure-point spectrum,
  we employ the equivalence theorem between pure-point diffraction and
  dynamical spectrum \cite[Thm.~7]{BL} and establish that the
  diffraction measure is pure point. This is clear for the covering
  model set by \cite[Thm.~9.4]{TAO}. Since the diffraction measure of
  a point set is not affected by a change of density $0$ to the point
  set, compare \cite[Sec.~8.8]{TAO}, it suffices to show that the
  point density of the covering model set determined from the CPS
  agrees with the true control point density.  This then means that
  the windows for the different tile types have no overlap of positive
  measure. The control point density from the CPS, by the standard
  uniform distribution property of the $\star$-image in internal
  space, see \cite{Moody} and \cite[Prop.~9.9]{TAO}, is given by the
  area $A$ of the window, divided by the volume $V$ of a unit cell of
  $\cL$. For the unit cell volume, we obtain
\[
  V \, = \, \myfrac{3}{4} \, \det
  \begin{pmatrix} 1 & \lambda \\ 1 & 8 - \lambda
  \end{pmatrix}^2 \cdot 81 \, = \, 3645 \ts ,
\]
where $81$ is the index of $\cL$ in the lift of $\ZZ[\xi,\lambda]$.
To determine the area $A$ of the window, one can convince oneself 
  (via Figure~\ref{fig:win}) that the window is a fundamental domain
of a triangular lattice, with a generating vector
$d = 31 + 4(\xi-\lambda) - \lambda\xi$, which gives a window area of
\[
  A \, = \, \myfrac{\sqrt{3}}{2} \, |d|^2 \, = \,
  \myfrac{135 \sqrt{3}}{2} \, (8 - \lambda) \ts .
\]
The control point density from the CPS is then given by
\[
  \rho^{}_1 \, = \, \myfrac{A}{V} \, = \,
  \myfrac{(8-\lambda) \sqrt{3}}{54} \ts .
\]
  
This density $\rho^{}_1$ now has to be compared to the true control
point density, which can be computed as follows. The right
Perron--Frobenius eigenvector of the tile inflation matrix contains
the relative frequences of the different tile types. If we normalize
the sum of these frequences to $1$, we obtain the vector
\[
  \bs{f} \, = \, (8-\lambda, 8-\lambda, 63-8\lambda, 63-8\lambda,
       15\lambda-118, 15\lambda-118, 8-\lambda,
       14\lambda-110, 197-25\lambda) \ts .
\]
We see that there is a triple and two pairs of tiles with the same
frequency. These form the tile clusters which always occur together.
Taking the scalar product of the frequency vector and the vector of
tile areas computed with the self-similar tile edges, we obtain the
average tile area, which evaluates to $90 \sqrt{3}$. The tile density
then is $ \sqrt{3}/270$. However, we now have to remember that we were
counting only one control point per tile cluster, so that we must
multiply this density with the sum of the \emph{distinct} frequencies
in the frequency vector $\bs{f}$. This then gives a total control
point density of
\[
  \rho^{}_2 \, = \, \myfrac{5(8-\lambda) \sqrt{3}}{270} \, = \,
  \myfrac{(8-\lambda)\sqrt{3}}{54} \, = \,  \rho^{}_1 \ts .
\]
Hence, the two control point densities agree, which proves
the following result, and corroborates the result of the
overlap algorithm.

\begin{theorem}\label{thm:model}
  The control points of the CASPr tiling comprise a full-density subset
  of the\/ $5$-color regular model set defined by the window system
  from Figure~$\ref{fig:win}$. They are dynamically defined Rauzy
  fractals, whose boundaries have the Hausdorff dimension from
  Eq.~$\eqref{eq:Hausdorff}$.

  As such, the CASPr tiling has pure-point diffraction, with the
  Fourier module
\[
      L^{\circledast} \ts =  \, \pi^{}_{\mathrm{int}} (\cL^*) \ts .
\]
The latter agrees with the dynamical pure-point spectrum of the CASPr
tiling, when viewed as a dynamical system under the translation action
of\/ $\RR^2$. This system is strictly ergodic and has continuously
representable eigenfunctions.  \qed
\end{theorem}

\begin{remark}\label{rem:dual}
  Let us expand on the role of $L^{\circledast}$ within the number
  field $K = \QQ (\alpha)$ from Remark~\ref{rem:max-order}. Observe
  first that the dual module of the non-maximal order
  $\cO = \ZZ[\xi,\lambda]$ is
\[
  \ZZ [ \xi,\lambda ]^* \, = \: \myfrac{\ii \sqrt{5}}{15}
  \, \ZZ [ \xi, \lambda ] \ts ,
\]
where $\ZZ [ \xi, \lambda ]$ is a submodule of its dual of index
$45^2$. The dual of a $\ZZ$-module $M \subset K$ is defined as
$M^* = \{ y \in K : x . y + (x.y)^{\prime} \in \ZZ \text{ for all }
x\in M\}$, with $x.y \defeq \frac{1}{2}( \bar{x} y + x \bar{y} )$.
Here, $(.)^{\prime}$ is the non-trivial algebraic conjugation in
$\QQ( \sqrt{15}\,)$, given by $\sqrt{15} \mapsto -\sqrt{15}$.  Note
that $x.y$ lies in this field for all $x,y\in K$, because
$x.y = \mathrm{Re} (\bar{x}y) \in K \nts \cap \RR = \QQ(\sqrt{15}\,)$,
which is the maximal real subfield of $K$. This form is $\RR$-linear
and designed to match the duality notion in Euclidean $4$-space,
relative to the standard inner product, as this fits best to the use
of the Fourier transform.

Further, the dual of the maximal order is
$\cO^{*}_{K} = \frac{\sqrt{15}}{15}\ts \cO^{}_{K}$. In fact, one has
\[
  L \, \overset{9^2}{\subset} \, \cO \, \overset{3}{\subset} \,
  \cO^{}_{K} \, \overset{15^2}{\subset} \, \cO^{*}_{K} =
  \myfrac{\sqrt{15}}{15} \ts \cO^{}_{K} \, \overset{3}{\subset} \,
  \cO^{*} = \myfrac{\ii\sqrt{5}}{15}\ts \cO \, \overset{9^2}{\subset} \,
  L^{*} ,
\]
again keeping track of the indices, and all matches up so that
$L^{\circledast} = L^*$ in this sense, and so that
$L^{\circledast} = \pi (\cL^*)$, where $\cL^*$ is the (standard) dual
of the lattice $\cL$ in $\CC^2 \simeq \RR^4$. The justification of
this computation comes from the dual CPS, which (in our case) can be
combined with the original CPS because $\CC$ and $\CC^2$ are self-dual
as locally compact Abelian groups.

We still need a good way to express $L^*$.  In
Remark~\ref{rem:return}, we identified the return module as the
non-principal ideal $(g^{}_{1}, g^{}_{3})$ of $\ZZ[\xi, \lambda ]$.
Now, if $\mathrm{N} (z)$ denotes the field norm of $z\in K$, which is
$\mathrm{N} (z) = z \bar{z} z^{\ts\prime} \bar{z}^{\ts\prime}$, one
can verify that the dual of a principal $\cO$-ideal is
\[
  (z)^* \, = \, \myfrac{1}{\bar{z}}\ts  \ZZ[\xi,\lambda ]^*
  \, = \, \myfrac{z z^{\ts\prime} \bar{z}^{\ts\prime}}{\mathrm{N} (z)}
  \, \myfrac{\ii \sqrt{5}}{15} \, \ZZ[\xi, \lambda ] \ts .
\]
Then, with $L = (g^{}_{1}) + (g^{}_{3})$ from Remark~\ref{rem:return},
some explicit computation leads to
\[
  L^{\circledast} \, = \; L^{*} \, = \;
  (g^{}_{1})^* \nts\cap (g^{}_{3})^* \, = \,
  \myfrac{\ii \sqrt{5}}{135} \ts L \ts .
\]
This shows that $L^{\circledast}$ is a non-principal, fractional
ideal, whose generators can now be given as (fractional) multiples of
the generators of $L$ from Remark~\ref{rem:return}, which is one of
the simplest ways to pin down the Fourier module.  \exend
\end{remark}

\section{The Spectre is MLD to a re-projected model
  set}\label{sec:deform}

Consider what happens when we take the same $4$-dimensional total
space as for the CASPr tiling, the same lattice, and the same
acceptance strip, only we vary the projection from $\RR^4$ to $\RR^2$.
Varying the projection moves each point in $\RR^2$ by a linear
function of its corresponding coordinates in perpendicular space. This
results in a topological conjugacy, since the perpendicular
coordinates are `weakly pattern equivariant', but not an MLD
equivalence, insofar as the perpendicular coordinates cannot be
determined exactly from the local pattern; we refer to
  \cite[Secs.~3.6 and 3.7]{KS} for a discussion of reprojections and
  pattern equivariance.

In particular, the tilings obtained by varying the projection in the
CASPr tiling form a (real) $4$-dimensional family that is topologically
conjugate to CASPr, but not MLD.  However, we have already determined
that, up to MLD equivalence, the set of tilings that are topologically
conjugate to CASPr is a connected $4$-dimensional family. The upshot
is that all tilings that are topologically conjugate to the CASPr are
MLD to reprojections of the CASPr control points. In particular, all
of the tilings that are obtained by varying the (complex) ratio $a/b$,
including the original Spectre tiling, are MLD to reprojections of
CASPr.

Let us illustrate this with two different reprojections. The first and
simpler one relies on the CASPr tiling being combinatorially
equivalent to a tiling of regular hexagons (compare
Figure~\ref{fig:combhex}), which have vertices in a simple hexagonal
lattice. We can therefore index the vertices of a CASPr tiling also
with vectors from that lattice.  The correspondence between the
$4$-dimensional CASPr indices and the new $2$-dimensional hexagonal
indices of the CASPr vertices defines a linear map from the CASPr
return module $L$ to the hexagonal lattice. If we choose the right
orientation and scale for the latter, this linear map is the
reprojection of the CASPr control points we are looking for. The right
scale is easily obtained by requiring that the regular hexagons are in
area equal to the average area of the CASPr tiles, which is known. In
the left panel of Figure~\ref{fig:hexrepro}, we show a patch of a
CASPr tiling reprojected in this way, along with its reprojected
control points. The CASPr tiles are considerably distorted under
reprojection, but clearly recognizable. For comparison, on the right
hand side of the figure, we show the corresponding regular hexagon
tiling, which is obviously MLD to the reprojected CASPr tiling: The
control points of the two tilings agree exactly.

\begin{figure}
  \centerline{
    \includegraphics[width=0.38\textwidth]{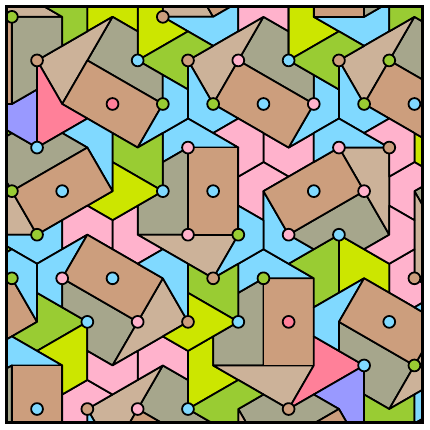}\hfil
    \includegraphics[width=0.38\textwidth]{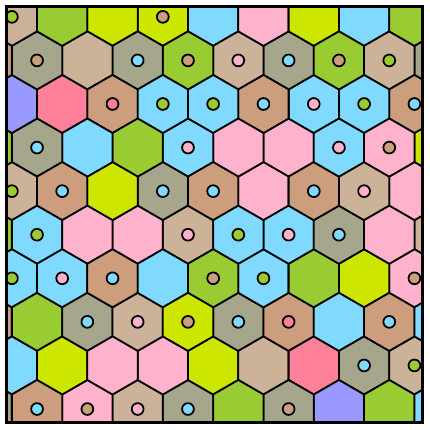}}
  \caption{Patch of a reprojected CASPr tiling (left) and the corresponding
    patch of a regular hexagon tiling which is MLD to it (right).
    The two patches have exactly the same control points. Note that,
    like for the CASPr, the control points are a bit away from their
    tiles. \label{fig:hexrepro}}
\end{figure}

\begin{figure}
  \centerline{\includegraphics[width=0.38\textwidth]{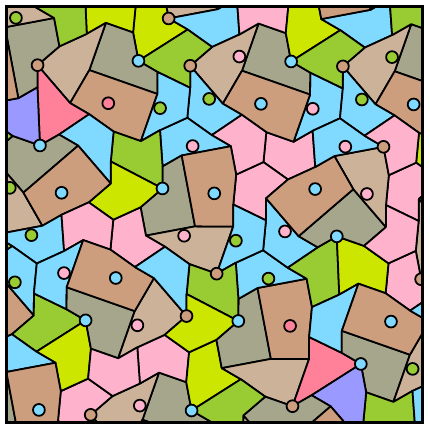}\hfil
              \includegraphics[width=0.38\textwidth]{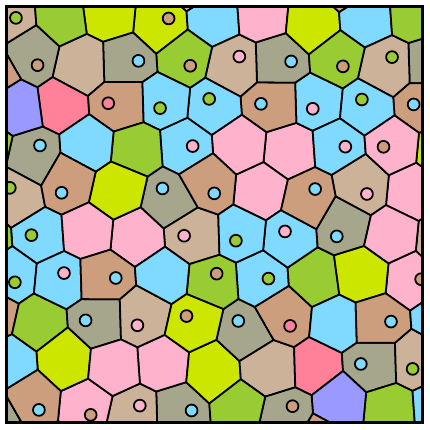}}
  \caption{Patch of a reprojected CASPr tiling (left) and the corresponding
    patch of a meta-tile tiling which is MLD to it (right).
    The two patches have exactly the same control points.
    \label{fig:metarepro}}
\end{figure}
 
Instead of regular hexagons, we can also use the meta-tiles of the
Hat-Turtle tiling (see \cite[Fig. 4.1]{Spectre}). These meta-tiles
also form a combinatorial hexagon tiling, whose tile edges take values
in another hexagon lattice, and we can index the CASPr vertices also
with respect to that hexagon lattice. This leads to yet another
reprojection map.  The reprojected CASPr tiling it produces is MLD to
a corresponding meta-tile tiling. Such a pair is shown in
Figure~\ref{fig:metarepro}. The reprojected CASPr tiles, shown on the
left, are now quite close to the true CASPr tiles.  Again, the two
patches are MLD to each other, with exactly the same control points.

\begin{coro}
  The Spectre tiling is MLD with a\/ $5$-color Meyer set that is a
  re-projection of the points from the cut-and-project description of
  the CASPr tiling from Theorem~$\ts\ref{thm:model}$.  In particular,
  the translation dynamical system of the Spectre tiling has the
    same pure-point spectrum with continuously representable
  eigenfunctions.  All this holds, in fact, for \emph{all}
  Spectre-like tilings.  \qed
\end{coro}

It is worth noting what changes and what does not change under
reprojection. As an abstract dynamical system, the tiling space does
not change. This implies that the dynamical spectrum, namely the
action of translation on $L^2(\Omega)$, does not change. This in turn
implies that the diffraction pattern generated by the control points
remains pure point, with the exact same Bragg peak locations.

However, the intensity of the Bragg peaks can vary. As we continuously
change CASPr into the Spectre tiling, adjusting the positions of the
control points, some Bragg peaks become more intense while others
fade. To the naked eye, the two diffraction patterns show rather
different features \cite{BGMM}, even though they are supported on the
same points. Various further details, including the deformation maps
and the diffraction images, will be given there.

\section*{Acknowledgements}

It is our pleasure to thank Markus Kirschmer for helpful discussions
on the number-theoretic structure. We thank two anonymous referees for
their thoughtful comments, which helped us to improve the
presentation. This work was supported by the German Research Council
(Deutsche Forschungsgemeinschaft, DFG) under SFB-TRR 358/1 (2023) --
491392403 and by the National Science Foundation under grant
DMS-2113468.  

\end{document}